\documentclass[reqno]{amsart}
\usepackage{amsthm,amsmath,amssymb}
\usepackage{stmaryrd,mathrsfs}
\usepackage[OT2,T1]{fontenc}
\usepackage{esint}

\allowdisplaybreaks

\newcommand{\ud}[0]{\,\mathrm{d}}

\newcommand{\dist}[0]{\operatorname{dist}}

\newcommand{\abs}[1]{|#1|}
\newcommand{\Babs}[1]{\Big|#1\Big|}

\newcommand{\Norm}[2]{\|#1\|_{#2}}

\newcommand{\BNorm}[2]{\Big\|#1\Big\|_{#2}}
\newcommand{\pair}[2]{\langle #1,#2 \rangle}
\newcommand{\Bpair}[2]{\Big\langle #1,#2 \Big\rangle}
\newcommand{\ave}[1]{\langle #1\rangle}

\newcommand{\bddlin}[0]{\mathscr{L}}

\newcommand{\BMO}[0]{\operatorname{BMO}}
\newcommand{\supp}[0]{\operatorname{supp}}
\newcommand{\loc}[0]{\operatorname{loc}}

\newcommand{\R}{\mathbb{R}}
\newcommand{\C}{\mathbb{C}}
\newcommand{\N}{\mathbb{N}}
\newcommand{\Z}{\mathbb{Z}}

\newcommand{\prob}[0]{\mathbb{P}}
\newcommand{\Exp}[0]{\mathbb{E}}
\newcommand{\D}[0]{\mathbb{D}}
\newcommand{\eps}[0]{\varepsilon}

\newcommand{\good}[0]{\operatorname{good}}
\newcommand{\bad}[0]{\operatorname{bad}}
\newcommand{\children}[0]{\operatorname{ch}}
\newcommand{\out}[0]{\operatorname{out}}
\newcommand{\inside}[0]{\operatorname{in}}
\newcommand{\near}[0]{\operatorname{near}}

\def\cyr{\fontencoding{OT2}\fontfamily{wncyr}\selectfont}
\DeclareTextFontCommand{\textcyr}{\cyr}
\newcommand{\sha}[0]{\textup{\textcyr{SH}}}

\swapnumbers \numberwithin{equation}{section}

\theoremstyle{plain}
\newtheorem{theorem}[equation]{Theorem}
\newtheorem{proposition}[equation]{Proposition}

\newtheorem{lemma}[equation]{Lemma}

\theoremstyle{definition}
\newtheorem{definition}[equation]{Definition}

\theoremstyle{remark}

\makeatletter
\@namedef{subjclassname@2010}{%
  \textup{2010} Mathematics Subject Classification}
\makeatother

\begin{document}

\title[The sharp weighted bound]{The sharp weighted bound for general Calder\'on--Zygmund operators} 

\author[T.~P.\ Hyt\"onen]{Tuomas P.\ Hyt\"onen}
\address{Department of Mathematics and Statistics, P.O.B.~68 (Gustaf H\"all\-str\"omin katu~2b), FI-00014 University of Helsinki, Finland}
\email{tuomas.hytonen@helsinki.fi}

\date{\today}



\maketitle

\begin{abstract}
For a general Calder\'on--Zygmund operator $T$ on $\R^N$, it is shown that 
\begin{equation*}
   \Norm{Tf}{L^2(w)}\leq C(T)\cdot\sup_Q\Big(\fint_Q w\cdot\fint_Q w^{-1}\Big)\cdot\Norm{f}{L^2(w)}
\end{equation*}
for all Muckenhoupt weights $w\in A_2$. This optimal estimate was known as the $A_2$ conjecture. 
A recent result of Perez--Treil--Volberg reduced the problem to a testing condition on indicator functions, which is verified in this paper.

The proof consists of the following elements: (i) a variant of the Nazarov--Treil--Volberg method of random dyadic systems with just one random system and completely without ``bad'' parts; (ii) a resulting representation of a general Calder\'on--Zygmund operator as an average of ``dyadic shifts''; and (iii) improvements of the Lacey--Petermichl--Reguera estimates for these dyadic shifts, which allow summing up the series in the obtained representation.
\end{abstract}

\section{Introduction}

Let $T\in\bddlin(L^2(\R^N))$ be a fixed Calder\'on--Zygmund operator, i.e., one with the integral representation
\begin{equation*}
  Tf(x)=\int_{\R^N}K(x,y)f(y)\ud y,\qquad x\notin\supp f,
\end{equation*}
for a kernel $K(x,y)$, defined for all $x\neq y$ on $\R^N\times\R^N$, and verifying the standard estimates $\displaystyle\abs{K(x,y)}\leq\frac{C}{\abs{x-y}^N}$ and
\begin{align*}
  \abs{K(x+h,y)-K(x,y)}+\abs{K(x,y+h)-K(x,y)}\leq \frac{C\abs{h}^{\alpha}}{\abs{x-y}^{N+\alpha}}
\end{align*}
for all $\abs{x-y}>2\abs{h}>0$ and some fixed $\alpha\in(0,1]$.  Let $w\in L^1_{\loc}(\R^N)$ be positive almost everywhere. It is classical that the Muckenhoupt condition
\begin{equation*}
  \Norm{w}{A_2}:=\sup_Q\fint_Q w\ud x\cdot\fint_Q w^{-1}\ud x <\infty,
\end{equation*}
where the supremum is taken over all cubes $Q\subset\R^N$, is both sufficient for the boundedness of all such $T$ on $L^2(w)$, and necessary for the $L^2(w)$-boundedness of some particular operators $T$, like the Hilbert transform for $N=1$.

Recently, the precise dependence of the $\bddlin(L^2(w))$ norm of Calder\'on--Zygmund operators on the Muckenhoupt characteristic $\Norm{w}{A_2}$ has attracted interest, and the following bound, optimal in general, has become known as the $A_2$ conjecture:
\begin{equation}\label{eq:A2conj}
  \Norm{Tf}{L^2(w)}\leq C(T)\,\Norm{w}{A_2}\Norm{f}{L^2(w)}.
\end{equation}
By the sharp form of Rubio de Francia's extrapolation theorem due to Dragi\v{c}evi\'c, Grafakos, Pereyra and Petermichl \cite{DGPP}, this implies the corresponding weighted $L^p$ bound,
\begin{equation}\label{eq:Apconj}
  \Norm{Tf}{L^p(w)}\leq C_p(T)\,\Norm{w}{A_p}^{\max\{1,1/(p-1)\}}\Norm{f}{L^p(w)},\qquad p\in(1,\infty),
\end{equation}
where
\begin{equation*}
   \Norm{w}{A_p}:=\sup_Q\fint_Q w\ud x\cdot\Big(\fint_Q w^{-1/(p-1)}\ud x\Big)^{p-1}.
\end{equation*}

Here is a brief description of past progress on this problem. It concentrates on the research on Calder\'on--Zygmund-type operators, for which the conjectured sharp bounds are given by \eqref{eq:A2conj} and \eqref{eq:Apconj}, but many other kinds of operators, sometimes with different dependence on the weight, have also been considered in the literature.
\begin{enumerate}
  \item Although not strictly a Calder\'on--Zygmund operator, the Hardy--Littlewood maximal operator $M$ is clearly closely related, and the sharp weighted line of research was opened by Buckley \cite{Buckley}, who proved \eqref{eq:A2conj} for $T=M$. (For $M$, the right power of $\Norm{w}{A_p}$ in \eqref{eq:Apconj} is $1/(p-1)$ for all $p\in(1,\infty)$.)
  
  \item Resolving a conjecture by Astala--Iwaniec--Saksman \cite[Eq.~(45)]{AIS} with implications to Beltrami equations, the case of the Beurling--Ahlfors transform $B\in\bddlin(L^2(\C))$ was first settled by Petermichl and Volberg \cite{PV}, and with an alternative proof by Dragi\v{c}evi\'c and Volberg \cite{DV}. Petermichl also obtained the sharp bounds for the Hilbert transform $H\in\bddlin(L^2(\R))$ \cite{Petermichl:Hilbert}, and then for the Riesz transforms $R_i\in\bddlin(L^2(\R^N))$ in arbitrary dimension $N\in\Z_+$~\cite{Petermichl:Riesz}. All these results relied on ad hoc representations based on specific symmetries of the operators in question, and Bellman function arguments tailor-made for each particular situation.
  
  \item A unified approach to the earlier results for $B$, $H$ and $R_i$ was found by Lacey, Petermichl and Reguera \cite{LPR}, who proved \eqref{eq:A2conj} for a general class of ``dyadic shifts'', from which all the mentioned operators may be obtained by suitable averaging. The original proof employed a two-weight inequality for dyadic shifts due to Nazarov, Treil and Volberg \cite{NTV:2weightHaar}. It was substantially simplified by Cruz-Uribe, Martell and P\'erez \cite{CMP}, based on a remarkable formula of Lerner \cite{Lerner:formula}, which gives very precise and useful information on a function in terms of its local mean oscillations.
  
  \item  Vagharshakyan \cite{Vagharshakyan} found a way of recovering all sufficiently smooth, odd, convolution-type Calder\'on--Zygmund operators in dimension $N=1$ from dyadic shifts, thereby proving \eqref{eq:A2conj} for all these operators. By a different method, Lerner \cite{Lerner:Ap} was able to estimate all  standard convolution-type operators in arbitrary dimension by controlling them in terms of Wilson's intrinsic square function \cite{Wilson}; however, this approach only gave  \eqref{eq:Apconj} for $p\in(1,\tfrac32]\cup[3,\infty)$.
  
  \item The conjecture \eqref{eq:A2conj} concerning a strong-type bound was reduced to proving the corresponding weak-type estimate (and even slightly less) by P\'erez, Treil and Volberg \cite{PTV}. Based on this reduction, the first confirmation of \eqref{eq:A2conj} for a general class of non-convolution operators,  but imposing heavy smoothness requirements on the kernels, was obtained by Lacey, Reguera, Sawyer, Uriarte-Tuero, Vagharshakyan and the author \cite{HLRSUV}.
\end{enumerate}

Altogether, the $A_2$ conjecture has now been verified in quite a number of cases. (Note that no two of the just mentioned results of Vagharshakyan \cite{Vagharshakyan}, Lerner \cite{Lerner:Ap}, and Lacey et al. \cite{HLRSUV} are strictly comparable.) And in this paper, the problem is completely solved. Besides, the proof is based on quite general metric--measure-theoretic objects (as opposed to the use of convolutions and regular wavelets in the preceding contributions), which makes it likely to extend to further situations like spaces of homogeneous type; see the discussion at the end of the paper.

\begin{theorem}
The estimate \eqref{eq:A2conj}, and hence \eqref{eq:Apconj}, holds for all Calder\'on--Zygmund operators $T\in\bddlin(L^2(\R^N))$, for all $N\in\Z_+$.
\end{theorem}

Just like the recent result of Lacey et al. \cite{HLRSUV}, the proof relies on the reduction of P\'erez, Treil and Volberg \cite{PTV}. For an arbitrary Calder\'on--Zygmund operator $T$, they proved that
\begin{align*}
  \Norm{T}{\bddlin(L^2(w))}
  &\leq C(T)\Big(\Norm{w}{A_2}+\sup_Q\frac{1}{w(Q)^{1/2}}\Norm{T^*(w1_Q)}{L^2(w^{-1})} \\
  &\phantom{\leq C(T)\Big(\Norm{w}{A_2}}+\sup_Q\frac{1}{w^{-1}(Q)^{1/2}}\Norm{T(w^{-1}1_Q)}{L^2(w)}\Big) \\
  &\leq C(T)\Big(\Norm{w}{A_2}+\Norm{T}{\bddlin(L^2(w),L^{2,\infty}(w))} +\Norm{T^*}{\bddlin(L^2(w^{-1}),L^{2,\infty}(w^{-1}))}\Big),
\end{align*}
where $w(Q):=\int_Q w\ud x$ and similarly with $w^{-1}$, and $T^*$ is the adjoint with respect to the unweighted $L^2$ duality. Thanks to the symmetry of $T$ and $T^*$ (both satisfy the same Calder\'on--Zygmund bounds), as well as of $w$ and $w^{-1}$ (both have the same $A_2$ characteristic), the first P\'erez--Treil--Volberg estimate above reduces the proof of the $A_2$ conjecture to showing that
\begin{equation}\label{eq:mainToShow}
   \Norm{T(w1_Q)}{L^2(w^{-1})}\leq C(T)\,\Norm{w}{A_2}\,w(Q)^{1/2}
\end{equation}
for all Calder\'on--Zygmund operators $T$. (The second P\'erez--Treil--Volberg estimate will not be used here; it is only recorded for the sake of pointing out the connection to weak-type bounds.)

This paper is concerned with the proof of \eqref{eq:mainToShow}. The Calder\'on--Zygmund operator $T$ will first be decomposed in terms of appropriate simpler operators. This was also the strategy of Lacey et al. \cite{HLRSUV}, where the decomposition was extracted from the proofs of the $T(1)$ theorems due to Beylkin--Coifman--Rokhlin \cite{BCR}, Figiel~\cite{Figiel}, and Xiang~\cite{Xiang}. However, the mentioned decomposition seems not to have been optimal for the $A_2$ conjecture, as summing up the weighted estimates for the simple operators required a high degree of smoothness on the kernel of $T$.

Thus, the first intermediate goal here is finding a better decomposition. And this is once again provided by the proof of a $T(1)$ theorem --- this time, the one for nonhomogeneous spaces due to Nazarov, Treil and Volberg \cite{NTV:Tb}. (A variant of the same proof, from a more recent Nazarov--Treil--Volberg preprint \cite{NTV:2weightHilbert}, is also behind the reduction of P\'erez, Treil and Volberg \cite{PTV}.) Recall that the basic philosophy of this proof is expanding an operator in terms of the Haar basis associated to a randomly chosen system of dyadic cubes; the part of the expansion living on so called ``good'' cubes can be directly estimated, and the remaining ``bad'' part can be forced to be an arbitrarily small fraction of the full operator norm. Thus the bound will be of the form
\begin{equation*}
  \Norm{T}{}\leq C_{\good}(r)+\eps_{\bad}(r)\Norm{T}{},
\end{equation*}
where $r$ is an adjustable parameter  in the definition of good and bad cubes; increasing $r$ will increase $C_{\good}(r)$ and decrease $\eps_{\bad}(r)$, and it suffices, in principle, to make $\eps_{\bad}(r)<1$. The problem is that, in the weighted case, the required size of $r$ would seem to depend on $w$. So even if one could prove the desired dependence $C_{\good}(r)\leq c(r)\Norm{w}{A_2}$ with $c(r)$ independent of $w$, this could be spoiled by the necessity of taking $r=r(w)$.

The solution to this problem is proving that, on average, the bad part becomes not only small but vanishing; in other words, a decomposition of an operator $T$ can be made by using Haar functions on good cubes only, with no error term whatsoever (Theorem~\ref{thm:goodMds}). This is an abstract result with no specific connection to weighted inequalities, and it will possible make the Nazarov--Treil--Volberg method of random dyadic systems more flexibly applicable to further questions. Another modification of the original randomisation argument is the use of only one random dyadic system, rather than two independent copies. In this way, there will be a stronger dyadic structure around, which is certainly a convenience, if not a necessity, for the subsequent considerations.

Once the full reduction to good cubes is available, the proof proceeds along the lines of the analysis of the good part in the Nazarov--Treil--Volberg $T(1)$ theorem \cite{NTV:Tb}, to extract several subseries of the Haar expansion, which are identified as new operators on their own right. These auxiliary operators are already implicit in the original Nazarov--Treil--Volberg argument \cite{NTV:Tb}, and their more explicit form was identified in my extension of their result to the vector-valued situation \cite{Hytonen:nonhomog}, where this explicit structure became more decisive. Here, it will be checked that these new operators are precisely the dyadic shifts in the generality defined by Lacey, Petermichl and Reguera \cite{LPR}. Thus, closing the circle with the pioneering sharp estimates for the classical integral transforms, it is proven here that all Calder\'on--Zygmund operators may be written as averages of dyadic shifts (Theorem~\ref{thm:repShifts}). In fact, and this technical issue will be important for the final steps of the proof, one only needs so called good dyadic shifts, where this goodness is closely related to the goodness of dyadic cubes.

The final task, then, is proving a version of the estimate \eqref{eq:mainToShow} for the good dyadic shifts in place of $T$. For individual shifts, this estimate has been established by Lacey--Petermichl--Reguera \cite{LPR}, with a simplified proof by Cruz-Uribe--Martell--P\'erez \cite{CMP}; however, their arguments give a dependence on certain parameters of the shift, which grows too rapidly to allow summing up the estimates in the series representation of $T$ in terms of these shifts. Appropriate improvements of these bounds will be established in the final part of the paper (Theorem~\ref{thm:mainEst}). Despite the elegance of the Cruz-Uribe--Martell--P\'erez argument \cite{CMP}, I did not manage to modify it for the required sharpness, and the new estimates will follow instead the general outline of the original Lacey--Petermichl--Reguera proof \cite{LPR}.

\subsection*{Acknowledgments}
I am indebted to Michael Lacey for getting me involved in the sharp weighted inequalities through my participation in the joint efforts \cite{HLRV,HLRSUV}.
The financial support of the Academy of Finland through projects 130166, 133264 and 218418 is gratefully acknowledged.

The major part of this research was carried out on the island of Koivula (N $59^{\circ} 52.922$, E $23^{\circ} 28.625$) in the Finnish archipelago. I would like to thank my parents for their kind hospitality.

\section{Preliminaries}

\subsection{Systems of dyadic cubes}
The standard dyadic system is
\begin{equation*}
  \mathscr{D}^{0}:=\bigcup_{k\in\Z}\mathscr{D}^{0}_k,\qquad
  \mathscr{D}^{0}_k:=\big\{ 2^{-k}\big([0,1)^N+m\big):m\in\Z^N\big\}.
\end{equation*}
For $I\in\mathscr{D}_k^0$ and a binary sequence $\beta=(\beta_j)_{j=-\infty}^{\infty}\in(\{0,1\}^N)^{\Z}$, let 
\begin{equation*}
  I\dot+\beta:=I+\sum_{j>k}\beta_j 2^{-j}.
\end{equation*}
Following Nazarov, Treil and Volberg \cite[Section 9.1]{NTV:Tb}, I will consider general dyadic systems of the form
\begin{equation*}
  \mathscr{D}=\mathscr{D}^{\beta}:=\{I\dot+\beta:I\in\mathscr{D}^0\}
  =\bigcup_{k\in\Z}\mathscr{D}^{\beta}_k.
\end{equation*}
Given a cube $I=x+[0,\ell)^N$, let
\begin{equation*}
   \children(I):=\{x+\eta\ell/2+[0,\ell/2)^N:\eta\in\{0,1\}^N\} 
\end{equation*}
denote the collection of dyadic children of $I$. Thus $\mathscr{D}^{\beta}_{k+1}=\bigcup\{\children(I):I\in\mathscr{D}^{\beta}_k\}$.

\subsection{Conditional expectations}
The local conditional expectation operators and their differences are denoted by
\begin{equation*}
  \Exp_I f:=1_I\ave{f}_I:= 1_I\fint_I f\ud x:= 1_I\frac{1}{\abs{I}}\int_I f\ud x,\qquad
  \D_I f:=\sum_{I'\in\children(I)}\Exp_{I'}f-\Exp_I f,
\end{equation*}
and then
\begin{equation*}
  \Exp_k^{\beta}f:=\sum_{I\in\mathscr{D}^{\beta}_k}\Exp_I f,\qquad
  \D_k^{\beta}f:=\sum_{I\in\mathscr{D}^{\beta}_k}\D_I f=\Exp_{k+1}^{\beta}f-\Exp_k^{\beta}f.
\end{equation*}
Often, the parameter $\beta$ will be understood from the context, and the superscript $\beta$ dropped from this notation.

For $f\in L^1_{\loc}(\R^N)$, Lebesgue's differentiation (or martingale convergence) theorem asserts that $\Exp_k f\to f$ almost everywhere, as $k\to\infty$. Since the $\Exp_k f$ are dominated by the Hardy--Littlewood maximal function $Mf$, this convergence also takes place in $L^2(w)$, as soon as $f\in L^2(w)$ and $w\in A_2$. This leads to the martingale difference decomposition
\begin{equation}\label{eq:mds}
\begin{split}
  f=\lim_{n\to\infty}\Exp_{n+1} f
  &=\Exp_m f+\lim_{n\to\infty}\sum_{k=m}^{n}\D_k f \\
  &=\sum_{I\in\mathscr{D}_m}\Exp_I f+\lim_{n\to\infty}\sum_{k=m}^{n}\sum_{I\in\mathscr{D}_k}\D_I f
\end{split}
\end{equation}
valid for any $m\in\Z$. The number $m$ will be considered fixed throughout most of the arguments. By abuse of notation, the operator $\D_I$ will be redefined as $\D_I+\Exp_I$ for $I\in\mathscr{D}_m$; then the identity \eqref{eq:mds} attains a simpler form without the first sum on the right.

\subsection{Haar functions}
Sometimes it is useful to write the operators $\D_I$ and $\Exp_I$ in terms of Haar functions $h_I^{\eta}$, $\eta\in\{0,1\}^N$, which satisfy
\begin{equation*}
  \supp h^{\eta}_I\subseteq I,\qquad h^{\eta}_I |_{I'}=\text{const}\quad\forall I'\in\children(I),\qquad\Norm{h^{\eta}_I}{\infty}\lesssim\abs{I}^{-1/2}
\end{equation*}
as well as
\begin{equation*}
  \int h^{\eta}_I h^{\theta}_I\ud x=\delta_{\eta\theta},\qquad h^0_I:=\abs{I}^{-1/2}1_I.
\end{equation*}
(The precise definition of $h^{\eta}_I$ for $\eta\neq 0$ may be done in a variety of ways, and is not important for the present purposes.) Then
\begin{equation*}
  \D_I f=\sum_{\eta\in\{0,1\}^N\setminus\{0\}}h_I^{\eta}\pair{h^{\eta}_I}{f},\qquad
  \Exp_I f=h_I^0\pair{h_I^0}{f}.
\end{equation*}

\subsection{Random dyadic systems; good and bad cubes}

Choosing a random dyadic system simply amounts to a random choice of the parameterising binary sequence $\beta=(\beta_j)_{j\in\Z}$, according to the canonical product probability measure $\prob_{\beta}$ on $(\{0,1\}^N)^{\Z}$ which makes the coordinates $\beta_j$ independent and identically distributed with $\prob_{\beta}(\beta_j=\eta)=2^{-N}$ for all $\eta\in\{0,1\}^N$. The symbol $\Exp_{\beta}$ denotes the expectation over the random variables $\beta_j$, $j\in\Z$; I will also use conditional expectations of the type $\Exp_{\beta}[\,\cdot\,|\beta_j:j\in\mathscr{J}]$, which means (as usual) that the variables $\beta_j$, $j\in\mathscr{J}$, are held fixed, and only those $\beta_j$ with $j\in\Z\setminus\mathscr{J}$ are averaged out.

Following Nazarov, Treil and Volberg, a dyadic cube $I$ will be called bad, if it is relatively close to the boundary of a much bigger dyadic cube. However, only one dyadic system rather than two will be considered at a time here, so $I$ will be compared with bigger cubes of the same dyadic system. More precisely, given parameters $r\in\Z_+$ and $\gamma\in(0,\tfrac12)$, a cube $I\in\mathscr{D}$ is said to be bad if there exists a $J\in\mathscr{D}$ with $\ell(J)\geq 2^r\ell(I)$ such that $\dist(I,\partial J)\leq\ell(I)^{\gamma}\ell(J)^{1-\gamma}$. Otherwise, $I$ is said to be good.

A pair of cubes $(I,J)\in\mathscr{D}\times\mathscr{D}$ is said to be good, if the smaller cube, say $I$, satisfies $\dist(I,\partial K)>\ell(I)^{\gamma}\ell(K)^{1-\gamma}$ for all $K\in\mathscr{D}$ with $2^r\ell(I)\leq\ell(K)\leq\ell(J)$. (Note that the condition is trivially true for $\ell(J)<2^r\ell(I)$.)

In the treatment of a Calder\'on--Zygmund kernel with H\"older exponent $\alpha$, the choice $\displaystyle\gamma:=\frac{\alpha}{2(N+\alpha)}$ is useful. In the sequel, some simple algebra involving this number will take place every now and then; however, the reader should not be misled to think that this precise choice is particularly critical. I have made this choice mainly because (i) it works and (ii) it is the one chosen by Nazarov--Treil--Volberg and used in several papers by now. However, any smaller $\gamma$ (depending only on $\alpha$ and $N$) would work equally well.

 The cubes of $\mathscr{D}^{\beta}$ will be often explicitly considered in the form $I\dot+\beta$, with $I\in\mathscr{D}^0$. Under this parameterisation, it is important to observe a fundamental independence property regarding goodness. First, by definition, the spatial position of
 \begin{equation*}
  I\dot+\beta:=I+\sum_{j:2^{-j}<\ell(I)}2^{-j}\beta_j
\end{equation*}
depends only on $\beta_j$ for $2^{-j}<\ell(I)$. Second, the relative position of $I\dot+\beta$ with respect to a bigger cube
\begin{equation*}
  J\dot+\beta=J+\sum_{j:2^{-j}<\ell(I)}2^{-j}\beta_j+\sum_{j:\ell(I)\leq 2^{-j}<\ell(J)}2^{-j}\beta_j
\end{equation*}
depends only on $\beta_j$ for $\ell(I)\leq 2^{-j}<\ell(J)$. Thus, the position and goodness of $I\dot+\beta$ are independent.

It is an immediate consequence of symmetry that the probability of a particular cube $I\in\mathscr{D}$ being bad is a number depending only on $r$, $\gamma$ and $N$. This number, $\pi_{\bad}$, maybe easily estimated as $\pi_{\bad}\lesssim_{\gamma,N} 2^{-r\gamma}$. (Thanks to the above mentioned independence of position and goodness, the computation is only slightly different from the case of two independent random systems considered in \cite{NTV:Tb}.) In much of the earlier work based on good and bad cubes, it was important that this number can be made as small as one likes by fixing $r$ large enough, and the requirements for its magnitude depended on the implicit constants in certain square function estimates. Here, it will only be necessary to choose $r$ large enough so that $\pi_{\bad}<1$, hence $\pi_{\good}:=1-\pi_{\bad}>0$, which is a simple geometric condition.

\subsection{Notational conventions}
The proof of the $A_2$ conjecture is now about to start. It will deal with a measure $w\in A_2$ and its dual measure $\sigma:=w^{-1}$, which has the same $A_2$ characteristic $\Norm{w}{A_2}=\Norm{\sigma}{A_2}$.

In the estimate to be proven, the precise dependence on the weight $w$ is decisive, so such dependence will always be indicated explicitly. On the other hand, the particular dependence on the properties of the arbitrary but fixed Calder\'on--Zygmund operator $T$ will be unimportant. Accordingly, the shorthand $A\lesssim B$ will be used for $A\leq C(T)\,B$, where $C(T)$ is any finite quantity depending at most on $T$. Here it is understood that the operator $T$ carries with it, in particular, the information on the dimension $N$ of the domain $\R^N$, as well as a H\"older exponent $\alpha$ and the related constant $C$ from the standard estimates verified by its kernel. The number $\gamma$ and a suitable choice of $r$ only depends on these quantities.

\section{The good martingale difference representation}

The representation result to be proven in this section is of an abstract nature, as the reader will easily realise, but the aim of its formulation below will not be the maximal generality, but rather the weighted application at hand in the present paper. Consider an integer $m$ fixed, while $n$ is a variable, which is taken to approach infinity. A summation over some intervals $I\in\mathscr{I}$, with the additional restriction that $2^{-n}<\ell(I)\leq 2^{-m}$, will be abbreviated as
\begin{equation*}
  \sum_{\substack{I\in\mathscr{I} \\ 2^{-n}<\ell(I)\leq 2^{-m}}}=:\sum_{I\in\mathscr{I}}^n.
\end{equation*}
It will not quite be true that only good cubes are needed in the representation; however, it can be arranged that the bigger cube in any required pairing
\begin{equation*}
  T_{JI}:=\pair{\D_J g}{T\D_I f}
\end{equation*}
is always good, and also the pair of cubes is good, meaning that the smaller cube stays away from the boundaries of the bigger cubes up to the size of the bigger cube, and this slightly restricted joint goodness will be enough for the subsequent considerations.

An intermediate form between the original random martingale difference decomposition of Nazarov, Treil and Volberg \cite{NTV:Tb} and the present formulation is found in the proof of my vector-valued nonhomogenenous $Tb$ theorem \cite{Hytonen:nonhomog}, although there still with two independent dyadic systems.

\begin{theorem}\label{thm:goodMds}
Let $T\in\bddlin(L^2(w))$ and $f\in L^2(w)$, $g\in L^2(\sigma)$ be compactly supported. Then the following representation is valid:
\begin{align*}
  &\pair{g}{Tf}\cdot\pi_{\good}^2 \\
  &=\lim_{n\to\infty}\Exp_{\beta}\Big[\sum_{\substack{I,J\in\mathscr{D}^0 \\ \ell(J)\geq \ell(I)}}^n
     T_{J\dot+\beta,I\dot+\beta} 
     \,\Exp_{\beta}[1_{\good(\beta)}(I\dot+\beta):\beta_j:2^{-j}<\ell(J)]\,1_{\good(\beta)}(J\dot+\beta) \\
  &\phantom{=\lim_{n\to\infty}\Exp_{\beta}}+\sum_{\substack{I,J\in\mathscr{D}^0 \\ \ell(J)< \ell(I)}}^n
     T_{J\dot+\beta,I\dot+\beta} 
     \,1_{\good(\beta)}(I\dot+\beta)\,\Exp_{\beta}[1_{\good(\beta)}(J\dot+\beta):\beta_j:2^{-j}<\ell(I)]\Big] \\
   &=\lim_{n\to\infty}\Exp_{\beta}\sum_{\substack{I,J\in\mathscr{D}^{\beta} \\ \textup{bigger cube good} \\ \textup{pair $(I,J)$ good}}}^n
     \pair{\D_J g}{T\D_I f}\cdot\pi_{IJ},
\end{align*}
where  $\pi_{IJ}\in[0,1]$ are the values of the conditional probabilities on the previous lines after reindexing the summation in terms of $\mathscr{D}^{\beta}$. The last summation condition is short hand for the requirement that the cube $J$ is good if $\ell(J)\geq\ell(I)$, the cube $I$ is good if $\ell(I)>\ell(J)$, and the pair of cubes $(I,J)$ is always good.
\end{theorem}

The rest of this section is concerned with the proof of this theorem. Observe first that
\begin{equation*}
  \pair{g}{Tf}=\pair{g}{T\Exp_n f}+\pair{g}{T(f-\Exp_n f)},
\end{equation*}
where the second term satisfies
\begin{equation*}
  \abs{\pair{g}{T(f-\Exp_n f)}}
  \leq\Norm{g}{L^2(\sigma)}\Norm{T}{\bddlin(L^2(w))}\Norm{f-\Exp_n f}{L^2(w)},
\end{equation*}
and the last factor is dominated by $C(w)\Norm{f}{L^2(w)}$ and tends to zero as $n\to\infty$. (At this point, the precise dependence of $C(w)$ on the weight is of not important.) By dominated convergence, also the expectation over the different dyadic systems of this quantity tends to zero as $n\to\infty$. Thus
\begin{equation*}
  \pair{g}{Tf}=\Exp_{\beta}\pair{g}{T\Exp_n f}+\eps_n,
\end{equation*}
where $\eps_n\to 0$ as $n\to\infty$. I keep using $\eps_n$ in this meaning; it need not be the exact same quantity on each occurrence. The compact support of $f$ ensures that $\Exp_n f$ is the finite sum
\begin{equation*}
  \Exp_n f=\sum_{I\in\mathscr{D}^{\beta}}^n\D_I f;
\end{equation*}
recall that $\D_I$ is abuse for $\D_I+\Exp_I$ when $\ell(I)=2^{-m}$.

Now I investigate the effect of the expectation $\Exp_{\beta}$ in more detail. Since $\D_{I\dot+\beta}f$ depends only on $\beta_j$ for $2^{-j}<\ell(I)$, whereas the goodness of $I\dot+\beta$ depends on the complementary parameters  $\beta_j$ for $2^{-j}\geq\ell(I)$, there holds by independence that
\begin{align*}
  \Exp_{\beta}[\pair{g}{T\D_{I\dot+\beta}f}1_{\good(\beta)}(I\dot+\beta)]
  &=\Exp_{\beta}[\pair{g}{T\D_{I\dot+\beta}f}]\cdot \Exp_{\beta}[1_{\good(\beta)}(I\dot+\beta)] \\
  &=\Exp_{\beta}[\pair{g}{T\D_{I\dot+\beta}f}]\cdot \pi_{\good}
\end{align*}
and hence
\begin{align*}
  \Exp_{\beta}\pair{g}{T\Exp_n f}
  &=\Exp_{\beta}\sum_{I\in\mathscr{D}^{\beta}}^n\pair{g}{T\D_{I}f}
  =\sum_{I\in\mathscr{D}^{0}}^n \Exp_{\beta}\pair{g}{T\D_{I\dot+\beta}f} \\
  &=\frac{1}{\pi_{\good}}\sum_{I\in\mathscr{D}^{0}}^n \Exp_{\beta}[\pair{g}{T\D_{I\dot+\beta}f}1_{\good}(\beta)(I\dot+\beta)] \\
  &=\frac{1}{\pi_{\good}} \Exp_{\beta}\sum_{I\in\mathscr{D}^{\beta}_{\good}}^n \pair{g}{T\D_{I}f}.
\end{align*}
Moreover, writing $g=\Exp_n g+(g-\Exp_n g)$, it follows that
\begin{align*}
  \sum_{I\in\mathscr{D}^{\beta}_{\good}}^n \pair{g}{T\D_{I\dot+\beta}f}
  =\sum_{I\in\mathscr{D}^{\beta}_{\good}}^n \sum_{J\in\mathscr{D}^{\beta}}^n \pair{\D_J g}{T\D_{I}f}
     +\Bpair{g-\Exp_n g}{T\sum_{I\in\mathscr{D}^{\beta}_{\good}}^n \D_I f},
\end{align*}
where the last term is dominated by
\begin{align*}
  \Norm{g-\Exp_n g}{L^2(\sigma)}\Norm{T}{\bddlin(L^2(w))} C(w)\Norm{f}{L^2(w)},
\end{align*}
and the first factor is bounded by $C(w)\Norm{g}{L^2(\sigma)}$ and tends to zero as $n\to\infty$. By dominated convergence again, it follows that
\begin{align*}
   \Exp_{\beta}\pair{g}{T\Exp_n f}
   =\frac{1}{\pi_{\good}} \Exp_{\beta}\sum_{I\in\mathscr{D}^{\beta}_{\good}}^n \sum_{J\in\mathscr{D}^{\beta}}^n \pair{\D_J g}{T\D_{I}f}+\eps_n.
\end{align*}

I keep manipulating the double sum, making use of the dependence of the various random quantities on the different parameters $\beta_j$, as well as basic properties of conditional expectations. There holds
\begin{align*}
  &\frac{1}{\pi_{\good}} \Exp_{\beta}\sum_{I\in\mathscr{D}^{\beta}_{\good}}^n \sum_{J\in\mathscr{D}^{\beta}}^n \pair{\D_J g}{T\D_{I}f} \\
  &=\frac{1}{\pi_{\good}} \Exp_{\beta}\Big(\sum_{I\in\mathscr{D}^{\beta}_{\good}}^n \sum_{\substack{J\in\mathscr{D}^{\beta} \\ \ell(J)\geq\ell(I)}}^n
      +\sum_{I\in\mathscr{D}^{\beta}_{\good}}^n \sum_{\substack{J\in\mathscr{D}^{\beta}\\ \ell(J)<\ell(I)}}^n\Big) \pair{\D_J g}{T\D_{I}f} =:A+B,
\end{align*}
and further
\begin{align*}
  A&=\frac{1}{\pi_{\good}}\sum_{\substack{I,J\in\mathscr{D}^{0} \\ \ell(J)\geq\ell(I)}}^n
       \Exp_{\beta}[\pair{\D_{J\dot+\beta} g}{T\D_{I\dot+\beta}f}\cdot 1_{\good(\beta)}(I\dot+\beta)] \\
  &=\frac{1}{\pi_{\good}}\sum_{\substack{I,J\in\mathscr{D}^{0} \\ \ell(J)\geq\ell(I)}}^n
       \Exp_{\beta}\big[\pair{\D_{J\dot+\beta} g}{T\D_{I\dot+\beta}f}\cdot \Exp_{\beta}\big[1_{\good(\beta)}(I\dot+\beta)|\beta_j:2^{-j}<\ell(J)\big]\big],
\end{align*}
where the first factor inside $\Exp_{\beta}$ only depends on $\beta_j$ for $2^{-j}<\ell(J)$, which allowed to replace the second factor by its conditional expectation with respect to these variables. Let then
\begin{equation*}
  \pi_{I\dot+\beta,\ell(J)}^{\beta}:=\Exp_{\beta}\big[1_{\good(\beta)}(I\dot+\beta)|\beta_j:2^{-j}<\ell(J)\big];
\end{equation*}
by definition, this conditional probability only depends on $\beta_j$ for $2^{-j}<\ell(J)$. As the goodness of $J\dot+\beta$ depends on the complementary variables $\beta_j$ for $2^{-j}\geq\ell(J)$, independence may be used again to write
\begin{align*}
  &\Exp_{\beta}\big[\pair{\D_{J\dot+\beta} g}{T\D_{I\dot+\beta}f}\cdot \pi_{I\dot+\beta,\ell(J)}^{\beta}\cdot  \Exp_{\beta}[1_{\good(\beta)}(J\dot+\beta)] \\
  &=\Exp_{\beta}\big[\pair{\D_{J\dot+\beta} g}{T\D_{I\dot+\beta}f}\cdot \pi_{I\dot+\beta,\ell(J)}^{\beta}   \cdot 1_{\good(\beta)}(J\dot+\beta)\big].
\end{align*}
Using this and recalling that $\Exp_{\beta}[1_{\good(\beta)}(J\dot+\beta)]=\pi_{\good}$, there holds
\begin{equation*}
  A =\frac{1}{\pi_{\good}^2}\Exp_{\beta}
  \sum_{\substack{I,J\in\mathscr{D}^{0} \\ \ell(J)\geq\ell(I)}}^n
       \pair{\D_{J\dot+\beta} g}{T\D_{I\dot+\beta}f} \times \pi_{I\dot+\beta,\ell(J)}^{\beta} \times 1_{\good(\beta)}(J\dot+\beta).
\end{equation*}
While the conditional probability $\pi_{I\dot+\beta,\ell(J)}^{\beta}$ is some number between $0$ and $1$ in general, it is important to notice a particular case when it is zero: this is when $I\dot+\beta$ is already bad with respect to some interval $K\in\mathscr{D}^{\beta}$ of length at most $\ell(J)$, in particular when $I\dot+\beta$ is bad with respect to $J\dot+\beta$. Hence, if $\pi_{I\dot+\beta,\ell(J)}^{\beta}>0$, then $(I\dot+\beta,J\dot+\beta)$ is good, and this additional restriction may be introduced without changing the value of the sum. Hence, reindexing in terms of $\mathscr{D}^{\beta}$ again,
\begin{align*}
  A &=\frac{1}{\pi_{\good}^2}
  \Exp_{\beta}\sum_{I\in\mathscr{D}^{\beta}}^n
  \sum_{\substack{J\in\mathscr{D}^{\beta}_{\good} \\ \ell(J)\geq\ell(I) \\ (I,J)\good}}^n
     \pair{\D_J g}{T\D_J f}\cdot\pi_{IJ},
\end{align*}
for certain numbers $\pi_{IJ}\in[0,1]$, whose dependence on $\beta$ is suppressed from the notation. 

In part $B$, simply by independence (the first factor depends on $\beta_j$ for $2^{-j}<\ell(I)$, the second on $\beta_j$ for $2^{-j}\geq\ell(I)$):
\begin{align*}
  B
   &=\frac{1}{\pi_{\good}}\sum_{\substack{I,J\in\mathscr{D}^{0} \\ \ell(J)<\ell(I)}}^n
       \Exp_{\beta}[\pair{\D_{J\dot+\beta} g}{T\D_{I\dot+\beta}f}\cdot 1_{\good(\beta)}(I\dot+\beta)] \\
   &=\frac{1}{\pi_{\good}}\sum_{\substack{I,J\in\mathscr{D}^{0} \\ \ell(J)<\ell(I)}}^n
       \Exp_{\beta}\pair{\D_{J\dot+\beta} g}{T\D_{I\dot+\beta}f}\cdot \Exp_{\beta}[1_{\good(\beta)}(I\dot+\beta)] \\
   &=\sum_{\substack{I,J\in\mathscr{D}^{0} \\ \ell(J)<\ell(I)}}^n
       \Exp_{\beta}\pair{\D_{J\dot+\beta} g}{T\D_{I\dot+\beta}f}
       =\Exp_{\beta}\sum_{\substack{I,J\in\mathscr{D}^{\beta} \\ \ell(J)<\ell(I)}}^n
       \pair{\D_{J} g}{T\D_{I}f}.
\end{align*}
Altogether, it has now been shown that
\begin{align*}
  \pair{g}{Tf}
  =\frac{1}{\pi_{\good}^2}
  &\Exp_{\beta}\sum_{I\in\mathscr{D}^{\beta}}^n
  \sum_{\substack{J\in\mathscr{D}^{\beta}_{\good} \\ \ell(J)\geq\ell(I) \\ (I,J)\good}}^n
     \pair{\D_J g}{T\D_I f}\cdot\pi_{IJ} \\
     &+\Exp_{\beta}\sum_{\substack{I,J\in\mathscr{D}^{\beta} \\ \ell(J)<\ell(I)}}^n
       \pair{\D_{J} g}{T\D_{I}f}+\eps_n,
\end{align*}
whereas also
\begin{align*}
  \pair{g}{Tf}
  =\pair{\Exp_n g}{T\Exp_n f}+\eps_n
  =\ &\Exp_{\beta}\sum_{I\in\mathscr{D}^{\beta}}^n
  \sum_{\substack{J\in\mathscr{D}^{\beta} \\ \ell(J)\geq\ell(I)}}^n
     \pair{\D_J g}{T\D_I f} \\
     &+\Exp_{\beta}\sum_{\substack{I,J\in\mathscr{D}^{\beta} \\ \ell(J)<\ell(I)}}^n
       \pair{\D_{J} g}{T\D_{I}f}+\eps_n,
\end{align*}
Comparing these equalities, it follows that
\begin{align*}
  \Exp_{\beta}\sum_{I\in\mathscr{D}^{\beta}}^n
  \sum_{\substack{J\in\mathscr{D}^{\beta} \\ \ell(J)\geq\ell(I)}}^n
     \pair{\D_J g}{T\D_I f}
   =\frac{1}{\pi_{\good}^2}
  \Exp_{\beta}\sum_{I\in\mathscr{D}^{\beta}}^n
  \sum_{\substack{J\in\mathscr{D}^{\beta}_{\good} \\ \ell(J)\geq\ell(I) \\ (I,J)\good}}^n
     \pair{\D_J g}{T\D_I f}\cdot\pi_{IJ}+\eps_n.
\end{align*}
A symmetric treatment, with the r\^oles of $I$ and $J$ reversed, also shows that
\begin{equation*}
    \Exp_{\beta}\sum_{I\in\mathscr{D}^{\beta}}^n
  \sum_{\substack{J\in\mathscr{D}^{\beta} \\ \ell(J)<\ell(I)}}^n
     \pair{\D_J g}{T\D_I f}
   =\frac{1}{\pi_{\good}^2}
  \Exp_{\beta}\sum_{J\in\mathscr{D}^{\beta}}^n
  \sum_{\substack{I\in\mathscr{D}^{\beta}_{\good} \\ \ell(J)<\ell(I) \\ (J,I)\good}}^n
     \pair{\D_J g}{T\D_I f}\cdot\pi_{IJ}+\eps_n
\end{equation*}
for some further numbers $\pi_{IJ}\in[0,1]$ related to conditional probabilities as before. Thus
\begin{align*}
  \pair{g}{Tf}&=
  \Exp_{\beta}\Big(\sum_{I\in\mathscr{D}^{\beta}}^n
  \sum_{\substack{J\in\mathscr{D}^{\beta} \\ \ell(J)\geq\ell(I)}}^n+
  \sum_{I\in\mathscr{D}^{\beta}}^n
  \sum_{\substack{J\in\mathscr{D}^{\beta} \\ \ell(J)<\ell(I)}}^n\Big)
     \pair{\D_J g}{T\D_I f}+\eps_n \\
   &=\frac{1}{\pi_{\good}^2}
  \Exp_{\beta}\Big(\sum_{I\in\mathscr{D}^{\beta}}^n
  \sum_{\substack{J\in\mathscr{D}^{\beta}_{\good} \\ \ell(J)\geq\ell(I) \\ (I,J)\good}}^n+
  \sum_{J\in\mathscr{D}^{\beta}}^n
  \sum_{\substack{I\in\mathscr{D}^{\beta}_{\good} \\ \ell(J)<\ell(I) \\ (J,I)\good}}^n\Big)
     \pair{\D_J g}{T\D_I f}\cdot\pi_{IJ}+\eps_n,
\end{align*}
which is the claim of the theorem.

\section{Decomposition into dyadic shifts}

With the martingale difference decomposition of the previous section as the starting point, the next goal is to express the operator $T$ as an average of fundamental building blocks called dyadic shifts. It is first in order to give a definition. Although expressed somewhat differently, it is essentially equivalent to that given by Lacey, Petermichl and Reguera \cite[Definition~1.5]{LPR}.

\begin{definition}\label{def:shifts}
A dyadic shift with parameters $(u,v)$ is an operator
\begin{equation*}
  \sha=\sum_{K\in\mathscr{D}}A_K,
\end{equation*}
where $\mathscr{D}$ is a dyadic system and each $A_K$ has the form
\begin{align*}
  A_K f(x) &:=\fint_K a_K(x,y)f(y)\ud y, \qquad\Norm{a_K}{\infty}\lesssim 1,\\
  a_K(x,y)&=\sum_{\substack{I\in\mathscr{D}; I\subseteq K \\ \ell(I)=2^{-u}\ell(K)}}\sum_{\substack{J\in\mathscr{D}; J\subseteq K \\ \ell(J)=2^{-v}\ell(K)}}
    \sum_{\eta,\theta\in\{0,1\}^N} a_{IJK}^{\eta\theta}h^{\theta}_J(x) h^{\eta}_I(y).
\end{align*}
A dyadic shift is called finite, if only finitely many $A_K$ are nonzero; bounded, if $\Norm{A_K f}{L^2}\lesssim\Norm{f}{L^2}$; and good, if
\begin{equation*}
  \dist(J,\partial K)\geq\tfrac12\ell(J)^{\gamma}\ell(K)^{1-\gamma}=2^{-1-v\gamma}\ell(K),
\end{equation*}
and similarly with $I$ in place of $J$, for all $I$ and $J$ for which some $a_{IJK}^{\eta\theta}$ is nonzero.
\end{definition}

Only finite shifts will be needed in the present considerations. This is a qualitative convenience, which ensures that no problems of convergence can arise; however, all the estimates will obviously have to be independent of the number of nonzero $A_K$. The goal of this section is to express a Calder\'on--Zygmung operator as a weak limit of averages of good, finite, uniformly bounded dyadic shifts:

\begin{theorem}\label{thm:repShifts}
Let $T\in\bddlin(L^2)$ be a bounded Calder\'on--Zygmund operator (hence also $T\in\bddlin(L^2(w))$ and $f\in L^2(w),g\in L^2(\sigma)$ be compactly supported. Then
\begin{align*}
  \pair{g}{Tf}
  =\lim_{n\to\infty}\Exp_{\beta}\sum_{u,v=r}^{\infty}2^{-\max(u,v)\alpha/2}\pair{g}{\sha^{uv}_{n\beta} f},
\end{align*}
where $\sha^{uv}_{n\beta}$ is a good finite dyadic shift adapted to the dyadic system $\mathscr{D}^{\beta}$, with parameters $(u,v)$, and $\Norm{\sha^{uv}_{n\beta}f}{L^2}\lesssim\Norm{f}{L^2}$ uniformly in $u,v,n$ and $\beta$.
\end{theorem}

Consider the representation of $\pair{g}{Tf}$ provided by the previous section and, for the moment, the part of the series with $\ell(I)\leq\ell(J)$. The summation conditions
\begin{equation}\label{eq:implicitSC}
  I\in\mathscr{D}^{\beta},\quad J\in\mathscr{D}_{\good}^{\beta},\quad (I,J)\ \good,\quad 2^{-n}<\ell(I)\leq\ell(J)\leq 2^{-m}
\end{equation}
will be implicitly in force until further notice; only additional restrictions in summation will be indicated explicitly.

I rearrange the summation following the well-known procedure from Nazarov, Treil and Volberg \cite{NTV:Tb}. (Also the subsequent analysis will closely follow \cite{NTV:Tb}, as well as \cite{Hytonen:nonhomog}. Some details will only be cited from these sources.)
\begin{equation*}
  \sum_{\ell(I)\leq\ell(J)}
  =\sum_{\dist(I,J)\geq\ell(I)}  
     +\sum_{\substack{\dist(I,J)<\ell(I)\\ \ell(I)<2^{-r}\ell(J)}}+\sum_{\substack{\dist(I,J)<\ell(I)\\ \ell(I)\geq 2^{-r}\ell(J)}} 
   =:\Sigma_{\out}+\Sigma_{\inside}+\Sigma_{\near}.
\end{equation*}
When $I$ and $J$ are taken from the same dyadic system, as is the case here, the condition $\dist(I,J)<\ell(I)\leq(J)$ in fact implies that $\dist(I,J)=0$.

\subsection{The term $\Sigma_{\out}$}
For the analysis of $\Sigma_{\out}$, recall the notion of the long distance \cite[Definition~6.3]{NTV:Tb}
\begin{equation*}
  D(I,J):=\ell(I)+\dist(I,J)+\ell(J)
\end{equation*}
as well as the integer-valued function \cite[end of Section~5]{Hytonen:nonhomog}
\begin{equation*}
  \theta(j):=\Big\lceil\frac{j\gamma+r}{1-\gamma}\Big\rceil.
\end{equation*}
Then
\begin{equation*}
  \Sigma_{\out}=\sum_{i=0}^{\infty}\sum_{j=0}^{\infty}\sum_{\substack{\ell(J)=2^i\ell(I)\\ 2^{j}<D(I,J)/\ell(J)\leq 2^{j+1}}}
  =:\sum_{i,j}\sigma_{\out}^{ij}.
\end{equation*}
For $I$ and $J$ appearing in $\sigma_{\out}^{ij}$, using the goodness of $J$, one can readily check \cite[a few lines after (7.5)]{Hytonen:nonhomog}, that $J\subseteq I^{(i+j+\theta(j))}$, where $I^{(k)}$ indicates the $k$ generations older dyadic ancestor of $I$: the unique $I^{(k)}\in\mathscr{D}$ with $I^{(k)}\supseteq I$ and $\ell(I^{(k)})=2^k\ell(I)$. Thus, taking $K:=I^{(i+j+\theta(j))}\in\mathscr{D}^{\beta}$ as a new auxiliary summation variable, one can write
\begin{equation}\label{eq:sigmaOutij}
  \sigma_{\out}^{ij}=\sum_{K\in\mathscr{D}^{\beta}}\sum_{\substack{J\in\mathscr{D}^{\beta}_{\good}; J\subseteq K \\ \ell(J)=2^{-j-\theta(j)}\ell(K)}}\
    \sum_{\substack{I\in\mathscr{D}^{\beta}; I\subseteq K; (I,J)\good \\ \ell(I)=2^{-i-j-\theta(j)}\ell(K)\\ 2^j<D(I,J)/\ell(J)<2^{j+1}}}
    =:\sum_{K\in\mathscr{D}^{\beta}}\sigma_K^{ij}.
\end{equation}

The next task is to check that $\sigma_K^{ij}$ is of the form $\pair{g}{A_K f}$. Recalling the suppressed summands $\pair{\D_J g}{T\D_I f}$ and invoking the Haar functions,
\begin{equation*}
  \sigma_K^{ij}=\sum_{I,J}\sum_{\eta,\theta}\pair{g}{h^{\theta}_J}\cdot\pair{h^{\theta}_J}{Th^{\eta}_I}\cdot\pair{h^{\eta}_I}{f}\cdot\pi_{IJ},
\end{equation*}
where the summation conditions on $I,J$ are as in \eqref{eq:sigmaOutij}, while $\eta,\theta$ run over $\{0,1\}^N\setminus\{0\}$, except possibly when $I\in\mathscr{D}_m$ or $J\in\mathscr{D}_m$ in which case also the noncancellative Haar functions $h^0_I$ or $h^0_J$ are allowed. Also recall that $\pi_{IJ}\in[0,1]$; no further properties of these conditional probabilities will be needed in the treatment of this part of the sum. For the coefficient $\pair{h^{\theta}_J}{Th^{\eta}_I}$, standard kernel estimates and the goodness of $I$ in the case that $\ell(I)<2^{-r}\ell(J)$ give \cite[Lemmas 6.1 and 6.4]{NTV:Tb}
\begin{align*}
  \abs{\pair{h^{\theta}_J}{Th^{\eta}_I}}
  &\lesssim\frac{\ell(I)^{\alpha}}{\dist(I,J)^{N+\alpha}}\Norm{h^{\theta}_J}{1}\Norm{h^{\eta}_I}{1} \\
  &\lesssim\frac{\ell(I)^{\alpha/2}\ell(J)^{\alpha/2}}{D(I,J)^{N+\alpha}}\abs{J}^{1/2}\abs{I}^{1/2} \\
  &\lesssim 2^{-i\alpha/2}2^{-j\alpha+j\gamma N/(1-\gamma)}\frac{\abs{J}^{1/2}\abs{I}^{1/2}}{\abs{K}}.
\end{align*}
The above estimate depends on the fact that the Haar function $h^{\eta}_I$ related to the smaller cube $I$ is a cancellative one. Since the noncancellative Haar functions only appear on generation $m$, the claimed fact could only fail if both $\ell(I)=\ell(J)=2^{-m}$. But one can choose $m$ so small (i.e., large negative) that at most $2^N$ cubes of length $2^{-m}$ intersect the union of the supports of $f$ and $g$. Then all relevant pairs of cubes with $\ell(I)=\ell(J)=2^{-m}$ are less than their common sidelength apart, and hence they will fall into the term $\Sigma_{\near}$.

Writing
\begin{equation*}
  \alpha^{ij\eta\theta}_{IJ}:=2^{i\alpha/2}2^{j[\alpha-\gamma N/(1-\gamma)]}\cdot\pair{h^{\theta}_J}{Th^{\eta}_I}\lesssim \abs{I}^{1/2}\abs{J}^{1/2}/\abs{K},
\end{equation*}
there holds
\begin{equation*}
  \sigma_{\out} =\sum_{i,j=0}^{\infty}2^{-i\alpha/2}2^{-j[\alpha-\gamma N/(1-\gamma)]}\pair{g}{\sha^{ij}_{\out}f},
\end{equation*}
where the promised dyadic shifts $\sha^{ij}_{\out}$ are explicitly given by
\begin{align*}
  \sha^{ij}_{\out}f 
  :=\sum_{K\in\mathscr{D}^{\beta}}
   \sum_{\substack{I\in\mathscr{D}^{\beta},J\in\mathscr{D}^{\beta}_{\good};\ I,J\subseteq K \\ \ell(J)=2^{i}\ell(I)=2^{-j-\theta(j)}\ell(K) \\ 2^j<D(I,J)/\ell(J)\leq 2^{j+1}}} 
      \sum_{\eta,\theta}h^{\theta}_J \alpha^{\eta\theta}_{IJ}\pair{h^{\eta}_I}{f}
   =:\sum_{K\in\mathscr{D}}A_K^{ij}f.
\end{align*}
The persistent summation conditions \eqref{eq:implicitSC} and the goodness of $(I,J)$ may be incorporated by simply defining some of the coefficients $\alpha^{\eta\theta}_{IJ}$ to be zero. From the estimate $\abs{\alpha^{\eta\theta}_{IJ}}\lesssim \abs{I}^{1/2}\abs{J}^{1/2}/\abs{K}$ and the size and support properties of the Haar functions, it follows that $A_K^{ij}$ is an averaging operator,
\begin{equation*}
  A_K^{ij}f(x)=\fint_K a^{ij}_K(x,y)f(y)\ud y,\qquad\Norm{a^{ij}_K}{\infty}\lesssim 1.
\end{equation*}

One needs to check that $\sha^{ij}_{\out}$ is a good shift. If $J\in\mathscr{D}^{\beta}_{\good}$ appears in $A_K$, it is immediate from the goodness of $J$ that $\dist(J,\partial K)\geq\ell(J)^{\gamma}\ell(K)^{1-\gamma}$. For $I$, one can argue as follows:
\begin{align*}
  \dist(I,\partial K)
  \geq\dist(J,\partial K)-D(I,J)
  \geq\ell(J)^{\gamma}\ell(K)^{1-\gamma}-2^{j+1}\ell(J),
\end{align*}
and $\ell(J)=\ell(J)^{\gamma}\ell(J)^{1-\gamma}=\ell(J)^{\gamma}(2^{-j-\theta(j)}\ell(K))^{1-\gamma}$; hence
\begin{align*}
  \dist(I,\partial K) &\geq\ell(J)^{\gamma}\ell(K)^{1-\gamma}\big(1-2^{j+1}2^{-j(1-\gamma)-(j\gamma+r)}\big) \\
  &\geq\ell(J)^{\gamma}\ell(K)^{1-\gamma}\big(1-2^{1-r})\geq\tfrac12\ell(J)^{\gamma}\ell(K)^{1-\gamma},
\end{align*}
and $\ell(J)\geq\ell(I)$. 

Now each individual $\sha^{ij}_{\out}$ is seen to be of the required form, but the parameterisation of the series still different from the one stated in the theorem. Thus, let
\begin{equation*}
  v:=j+\theta(j)=\frac{j}{1-\gamma}+O(1),\qquad u:=i+v
\end{equation*}
so that $\ell(I)=2^{-u}\ell(K)$ and $\ell(J)=2^{-v}\ell(K)$ for all $I,J$ appearing in $A_K$. Then
\begin{equation*}
  2^{-j[\alpha-\gamma N/(1-\gamma)]}\lesssim 2^{-v[\alpha(1-\gamma)-N\gamma]}=2^{-v\alpha/2},
\end{equation*}
and hence
\begin{equation*}
  2^{-i\alpha/2}2^{-j[\alpha-\gamma N/(1-\gamma)]}\lesssim 2^{-(i+v)\alpha/2}=2^{-u\alpha/2}=2^{-\max(u,v)\alpha/2}.
\end{equation*}
This completes the treatment of $\Sigma_{\out}$.

\subsection{The term $\Sigma_{\inside}$} The first basic observation is that the conditions $\dist(I,J)<\ell(I)<2^{-r}\ell(J)$ and the goodness of $I$ imply that in fact $I$ must be fully contained in (and even deep inside) one of the children $J'\in\children(J)$ of $J$. On this set, $\D_J g$ takes a constant value $\ave{\D_J g}_{J'}=\ave{\D_J g}_I$. Then, for $I,J$ appearing in $\Sigma_{\inside}$, a paraproduct can be extracted, as usual,
\begin{align*}
  \pair{\D_J g}{T\D_I f}
  &=\pair{1_{(J')^c}\D_J g}{T\D_I f}+\ave{\D_J g}_{J'}\pair{1_{J'}}{T\D_I f} \\
  &=\pair{1_{(J')^c}(\D_J g-\ave{\D_J g}_{J'})}{T\D_I f}+\ave{\D_J g}_{J'}\pair{1}{T\D_I f} \\
  &=\sum_{\eta,\theta}\pair{g}{h^{\theta}_J}\pair{1_{(J')^c}(h^{\theta}_J-\ave{h^{\theta}_J}_{J'})}{Th^{\eta}_I}\pair{h^{\eta}_I}{f}
      +\ave{\D_J g}_{I}\pair{T^*1}{\D_I f}.
\end{align*}
The coefficients in the first term satisfy (cf. \cite[Lemma 7.3]{NTV:Tb} or \cite[Lemma 8.3]{Hytonen:nonhomog})
\begin{align*}
  \abs{\pair{1_{(J')^c}(h^{\theta}_J-\ave{h^{\theta}_J}_{J'})}{Th^{\eta}_I}}
  &\lesssim\Big(\frac{\ell(I)}{\ell(J)}\Big)^{\alpha/2}\Big(\frac{\Norm{h^{\theta}_J}{1}}{\abs{J}}+\abs{\ave{h^{\theta}_J}_{J'}}\Big)\Norm{h^{\eta}_I}{1} \\
  &\lesssim\Big(\frac{\ell(I)^{\alpha}}{\ell(J)}\Big)^{\alpha/2}\Big(\frac{\abs{I}}{\abs{J}}\Big)^{1/2}=2^{-i\alpha/2}\Big(\frac{\abs{I}}{\abs{J}}\Big)^{1/2}
\end{align*}
for $\ell(I)=2^{-i}\ell(J)$. Altogether then, 
\begin{equation*}
  \Sigma_{\inside}=\sum_{i=r+1}^{\infty}2^{-i\alpha/2}\pair{g}{\sha_{\inside}^i f}+\sum_I \pair{T^*1}{\D_I f}
    \sum_{\substack{J\supset I\\ \ell(J)>2^r\ell(I)}} \ave{\D_J g}_{I}\cdot\pi_{IJ},
\end{equation*}
where the new sequence of dyadic shifts is given by
\begin{align*}
  \sha^i_{\inside}f &=\sum_{J\in\mathscr{D}^{\beta}_{\good}}\sum_{\substack{I\in\mathscr{D}^{\beta},\ I\subset J\\ \ell(I)=2^{-i}\ell(J)}}\sum_{\eta,\theta}
        h^{\theta}_J \alpha^{\eta\theta}_{IJ}\pair{h^{\eta}_I}{f} \\
        &=\sum_{K\in\mathscr{D}^{\beta}}\sum_{\substack{J\in\mathscr{D}^{\beta}_{\good},\ J\subset K\\ \ell(J)=2^{-r}\ell(K)}}
           \sum_{\substack{I\in\mathscr{D}^{\beta},\ I\subset J\\ \ell(I)=2^{-i-r}\ell(K)}}\sum_{\eta,\theta}
        h^{\theta}_J \alpha^{\eta\theta}_{IJ}\pair{h^{\eta}_I}{f}
        =:\sum_{K\in\mathscr{D}^{\beta}}A_K f.
\end{align*}
The middle equality follows by simply introducing the new summation variable $K:=J^{(r)}$. Again, the implicit summation conditions \eqref{eq:implicitSC} are also in force, but may be suppressed by defining some of the $\alpha^{\eta\theta}_{IJ}$ as zero. The coefficients satisfy $\abs{\alpha^{\eta\theta}_{IJ}}\lesssim(\abs{I}/\abs{J})^{1/2}$ which, in combination with the properties of the Haar functions, shows that
\begin{equation*}
  A^i_K f(x)=\fint_K a^i_K(x,y)f(y)\ud y,\qquad\Norm{a^i_K}{\infty}\lesssim 1.
\end{equation*}

It is further clear that $\sha^i_{\inside}$ is a shift with parameters $(u,v)=(i+r,r)$, and $2^{-i\alpha/2}\lesssim 2^{-\max(u,v)\alpha/2}$, since $r$ is a fixed number. The goodness conditions for the shift follow for $J$ directly from the the goodness of $J$, and for $I$ from the fact that $I\subset J$ so that $\dist(I,\partial K)\geq\dist(J,\partial K)\geq\ell(J)^{\gamma}\ell(K)^{1-\gamma}$.

\subsection{The paraproduct}
It is time to treat the part of $\Sigma_{\inside}$ which was left over after the extraction of the shifts $\sha^i_{\inside}$ above. Making the suppressed summation conditions explicit, it is
\begin{equation*}
  \sum_{I\in\mathscr{D}^{\beta}}^n\pair{T^*1}{\D_I f}
  \sum_{\substack{J\in\mathscr{D}_{\good}^{\beta},\ J\supset I \\ \ell(J)>2^r\ell(I)\\ (I,J)\good}}^n\ave{\D_J g}_I\cdot\pi_{IJ},
\end{equation*}
where the conditions that $J\supset I$ and $(I,J)$ be good may as well be dropped from the last sum, since otherwise $\ave{D_J g}_I=0$ or $\pi_{IJ}=0$.  Now I resort to the fact that it is the expectation $\Exp_{\beta}$ of this quantity which ultimately matters, and it is also important to recall the precise definition of the numbers $\pi_{IJ}$. (A predecessor of the following computation is found in \cite[Section 9]{Hytonen:nonhomog}.) Abbreviating temporarily $T_{IJ}:=\pair{T^*1}{\D_I f}\,\ave{\D_J g}_I$, this leads to the expression
\begin{align*}
  &\Exp_{\beta}\sum_{I\in\mathscr{D}^{\beta}}^n\pair{T^*1}{\D_I f}
  \sum_{\substack{J\in\mathscr{D}^{\beta}_{\good}\\ \ell(J)>2^r\ell(I)}}^n\ave{\D_J g}_I\cdot\pi_{IJ} \\
  &=\Exp_{\beta}\sum_{\substack{I,J\in\mathscr{D}^0\\ \ell(J)>2^{r}\ell(I)}}^n
       T_{I\dot+\beta,J\dot+\beta}\,  \Exp_{\beta}[1_{\good(\beta)}(I\dot+\beta)|\beta_j:2^{-j}<\ell(J)]\, 1_{\good(\beta)}(J\dot+\beta) \\ 
  &=\sum_{\substack{I,J\in\mathscr{D}^0\\ \ell(J)>2^{r}\ell(I)}}^n
       \Exp_{\beta}\big[T_{I\dot+\beta,J\dot+\beta} \, \Exp_{\beta}[1_{\good(\beta)}(I\dot+\beta)|\beta_j:2^{-j}<\ell(J)]\big] \,  \Exp_{\beta}[1_{\good(\beta)}(J\dot+\beta)] \\
  &=\sum_{\substack{I,J\in\mathscr{D}^0\\ \ell(J)>2^{r}\ell(I)}}^n
       \Exp_{\beta}\big[T_{I\dot+\beta,J\dot+\beta} \, 1_{\good(\beta)}(I\dot+\beta)\big] \,  \pi_{\good} \\
  &=\pi_{\good}\Exp_{\beta}\sum_{I\in\mathscr{D}^{\beta}_{\good}}^n \pair{T^*1}{\D_I f}
       \sum_{\substack{J\in\mathscr{D}^{\beta},\ J\supset I\\ \ell(J)>2^{r}\ell(I)}}^n\ave{\D_J g}_I,
\end{align*}
where the natural condition that $J\supset I$ was reimposed to avoid unnecessary zeros in the sum.
 
In the inner sum, $\ave{\D_J g}_I=\ave{g}_{J'}-\ave{g}_J$, where $I\supset J'\in\children(J)$ and $\ell(J)<2^{-m}$. Recalling the abuse of notation when $\ell(J)=2^{-m}$, when $\D_J$ in fact stands for $\D_J+\Exp_J$, there holds $\ave{(\D_J+\Exp_J) g}_I=\ave{g}_{J'}$ in this case. Thus the summation over $J$ (if nonempty) is telescopic, and collapses to $\ave{g}_{I^{(r)}}$. For simplicity of notation, let $\ave{g}_{J}$ be abuse notation for zero in the case of an empty sum, i.e., when $\ell(J)\geq 2^{-m}$. After collapsing the telescope as explained, the computation continues by essentially reversing what was done above, but with the collapsed double sum: (A useful temporary abbreviation now is $T_{IJ}:=\pair{T^*1}{\D_I f}\,\ave{g}_{J}\,1_{\children^r(J)}(I)$, where the last factor is one if and only if $I\subset J$ with $\ell(I)=2^{-r}\ell(J)$.)
\begin{align*}
  &=\pi_{\good}\Exp_{\beta}\sum_{I\in\mathscr{D}^{\beta}_{\good}}^n\pair{T^*1}{\D_I f}\ave{g}_{I^{(r)}} \\
  &=\pi_{\good}\Exp_{\beta}\sum_{J\in\mathscr{D}^{\beta}}^n\ave{g}_J
    \sum_{\substack{I\in\mathscr{D}^{\beta}_{\good}, I\subset J\\ \ell(I)=2^{-r}\ell(J)}}^n\pair{T^*1}{\D_I f} \\
  &=\pi_{\good}\sum_{\substack{I,J\in\mathscr{D}^{0}\\ \ell(J)=2^{r}\ell(I)}}^n
       \Exp_{\beta}[T_{I\dot+\beta,J\dot+\beta}\, 1_{\good(\beta)}(I\dot+\beta)] \\
  &=\sum_{\substack{I,J\in\mathscr{D}^{0}\\ \ell(J)=2^{r}\ell(I)}}^n
       \Exp_{\beta}\big[T_{I\dot+\beta,J\dot+\beta}\,
           \Exp_{\beta}[1_{\good(\beta)}(I\dot+\beta)|\beta_j:2^{-j}<\ell(J)]\big]\,\Exp_{\beta}[1_{\good(\beta)}(J\dot+\beta)] \\
  &=\Exp_{\beta}\sum_{\substack{I,J\in\mathscr{D}^{0}\\ \ell(J)=2^{r}\ell(I)}}^n
       T_{I\dot+\beta,J\dot+\beta}\,\pi_{I\dot+\beta,\ell(J)}^{\beta}\,1_{\good(\beta)}(J\dot+\beta) \\
  &=\Exp_{\beta}\sum_{J\in\mathscr{D}^{\beta}_{\good}}^n\sum_{\substack{I\in\mathscr{D}^{\beta}, I\subset J \\ \ell(I)=2^{-r}\ell(J)}}^n
        \ave{g}_J\cdot\pair{T^*1}{\D_I f}\cdot\pi_{IJ}.
\end{align*}
In order to interpret this as an average of good dyadic shifts, one still needs to introduce the new summation variable $K:=J^{(r)}$, leading to
\begin{align*}
  &=\Exp_{\beta}\sum_{K\in\mathscr{D}^{\beta}}\sum_{\substack{J\in\mathscr{D}^{\beta}_{\good}, J\subset K\\ \ell(J)=2^{-r}\ell(K)}}^n
       \sum_{\substack{I\in\mathscr{D}^{\beta}, I\subset J \\ \ell(I)=2^{-2r}\ell(K)}}^n
        \ave{g}_J\cdot\pair{T^*1}{\D_I f}\cdot\pi_{IJ} \\
    &=:\Exp_{\beta}\Bpair{g}{\sum_{K\in\mathscr{D}^{\beta}}A_K f}=:\Exp_{\beta}\pair{g}{\Pi^*f},
\end{align*}
where $\Pi^*$ is a dual paraproduct operator. Note that the kernel of $A_K f(x)=\fint_K a_K(x,y)f(y)\ud y$ is
\begin{equation*}
  a_K(x,y)=\abs{K}\sum_{I,J}\sum_{\eta}\frac{1_J(x)}{\abs{J}}\cdot\pair{T^*1}{h^{\eta}_I}\cdot\pi_{IJ}\cdot h^{\eta}_I(y),
\end{equation*}
where the summation conditions are the same as above, and $\abs{\pair{T^*1}{h^{\eta}_I}}\lesssim\abs{I}^{1/2}$ since $T^*1\in\BMO$. As $\abs{K}/\abs{J}=2^{r}$, it follows that $\Norm{a_K}{\infty}\lesssim 1$, as required. Also, the goodness of $J$ ensures that $\dist(J,\partial K)\geq\ell(J)^{\gamma}\ell(K)^{1-\gamma}$, and the same estimate follows for $I$ simply because $I\subset J$. This completes the verification that $\Pi^*$ is a good dyadic shift with parameters $(u,v)=(2r,r)$.

\subsection{The term $\Sigma_{\near}$}
Here the summation conditions are $2^{-r}\ell(J)\leq\ell(I)\leq\ell(J)$ and $\dist(I,J)<\ell(I)$, which implies that in fact $\dist(I,J)=0$. Splitting the sum according to the value of $i=0,1,\ldots,r$ such that $\ell(I)=2^{-i}\ell(J)$, the goodness of $J$ implies that $J\subset K:=I^{(r+i)}$, which can be taken as a new summation variable.
\begin{equation*}
  \Sigma_{\near}=\sum_{i=0}^r\pair{g}{\sha_{\near}^i f},\qquad \sha_{\near}^i f:=\sum_{K\in\mathscr{D}'}A^i_K f,
\end{equation*}
where
\begin{align*}
   A_K^i f:=\sum_{\substack{J\in\mathscr{D}^{\beta}_{\good};\ J\subset K\\ \ell(J)=2^{-r}\ell(K)}}\
      \sum_{\substack{I\in\mathscr{D};\ I\subset K\\ \ell(I)=2^{-r-i}\ell(K)}}\sum_{\eta,\theta}
      h^{\theta}_J \alpha^{\eta\theta}_{IJ}\pair{h^{\eta}_I}{f}
\end{align*}
and, simply by the boundedness of $T$ on $L^2(\R^N)$,
\begin{equation*}
  \abs{\alpha^{\eta\theta}_{IJ}}
  =\abs{\pair{h^{\theta}_J}{Th^{\eta}_I}\cdot\pi_{IJ}}\lesssim\Norm{h^{\theta}_J}{2}\Norm{h^{\eta}_I}{2}=1.
\end{equation*}
Using the size of the Haar functions and the fact that both $I$ and $J$ are essentially of the same size as $K$, it follows that $A_K$ has the right size.

The goodness of $J$ implies that $\dist(J,\partial K)\geq\ell(J)^{\gamma}\ell(K)^{1-\gamma}$ and, using that $\dist(I,J)=0$,
\begin{align*}
  \dist(I,\partial K)
  &\geq\dist(J,\partial K)-\ell(J) \\
  &\geq \ell(J)^{\gamma}\ell(K)^{1-\gamma}(1-2^{-r(1-\gamma)})\geq\tfrac12\ell(J)^{\gamma}\ell(K)^{1-\gamma}.
\end{align*}
Thus $\sha^i_{\near}$ is a good dyadic shift with parameters $(u,v)=(r+i,r)$.

\subsection{Completion of the decomposition}
In the part of the sum martingale difference representation with $\ell(I)>\ell(J)$, one can perform completely analogous considerations as above on the dual side, leading to a series of pairings $\pair{\sha g}{f}$, where $\sha$ is a good dyadic shift. However, the definition of a good shift is self-dual, in the sense that $\sha^*$ satisfies all the conditions if and only if $\sha$ does. Hence, simply writing $\pair{\sha g}{f}=\pair{g}{\sha^* f}$ in each summand of the dual series, even this part attains the required form. As a curiosity, it may be observed that the part with $\ell(I)>\ell(J)$ gives shifts with parameters $(u,v)$ such that $u<v$, whereas $\ell(I)\leq\ell(J)$ gave $u\geq v$. Indeed, the adjoint of a shift with parameters $(u,v)$ is a shift with parameters $(v,u)$.

Theorem~\ref{thm:repShifts} still claims the finiteness and the uniform boundedness of all the appearing shifts $\sha$. The finiteness is clear from the fact that these shifts are constructed by reorganising the finite sums $\displaystyle\sum_{I,J\in\mathscr{D}^{\beta}}^n$ from the martingale difference representation. Concerning the uniform boundedness on (the unweighted!) $L^2$, this may be easily extracted from Nazarov, Treil and Volberg's proof of the nonhomogeneous $Tb$ theorem \cite{NTV:Tb}, in which this decomposition is implicitly performed. It is also not difficult to give a direct proof in the present homogeneous situation; however, somewhat different considerations are required for the cancellative shifts, which involve the noncancellative Haar functions on at most one level, and the paraproducts, where noncancellative Haar functions are present on all length-scales. But once this unweighted boundedness is known, the weighted estimates for the different shifts can be established in a uniform manner, without distinguishing the paraproducts from the other kinds of shifts. The proof of Theorem~\ref{thm:repShifts} is complete.

\section{Unweighted end-point estimate for the shifts}

The basic unweighted estimate for the dyadic shifts is the uniform (in the shift parameters) boundedness on $L^2$, which was made a part of Definition~\ref{def:shifts} above.
The next step is proving appropriate weak-type bounds in $L^1$. This is the same general strategy as in Lacey--Petermichl--Reguera \cite{LPR}; the novelty consists of improving the exponential dependence on the shift parameters to a linear one.

\begin{proposition}
A bounded dyadic shift with parameters $(u,v)$ maps $L^1$ into $L^{1,\infty}$ with norm $O(u)$.
\end{proposition}

\begin{proof}
This is a rather classical-style argument based on the Calder\'on--Zygmund decomposition. Given $f\in L^1(\R^N)$, let $g$ and $b$ be its good and bad parts with respect to height $\lambda$ and the dyadic system $\mathscr{D}$ related to the particular shift; i.e., $\displaystyle b=f-g=\sum_{L\in\mathscr{B}}b_L$ with $b_L:=1_L(f-\ave{f}_L)$, where $L\in\mathscr{B}\subset\mathscr{D}$ are the maximal dyadic cubes with $\fint_L \abs{f}\ud x>\lambda$. As usual
\begin{align*}
  \abs{\{\abs{\sha f}>\lambda\}}
  &\leq\abs{\{\abs{\sha g}>\tfrac12\lambda\}}+\abs{\{\abs{\sha b}>\tfrac12\lambda\}}, \\
  \abs{\{\abs{\sha g}>\tfrac12\lambda\}}
  &\leq 4\lambda^{-2}\Norm{\sha g}{2}^2\lesssim\lambda^{-2}\Norm{g}{2}^2\lesssim\lambda^{-1}\Norm{f}{1},
\end{align*}
and
\begin{equation*}
  \sha b=\sum_L\sha b_L=\sum_L\sum_K A_K b_L.
\end{equation*}
A necessary condition for $A_K b_L\neq 0$ is $K\cap L\neq\varnothing$, which means that $K\subseteq L$ or $K\supset L$. But, if $\ell(K)>2^u\ell(L)$, then the kernel $a_K(x,y)$ of $A_K$, as a function of $y$, is constant on all $I\in\mathscr{D}$ with $\ell(I)=\ell(L)$, and in particular on $L$. Since $\int b_L=0$, it follows that $A_K b_L=0$ also in this case. Thus
\begin{equation*}
  \sum_K A_K b_L=\sum_{K\subseteq L}A_K b_L+\sum_{i=1}^u A_{L^{(i)}} b_L.
\end{equation*}
The first sum is supported on $L$, and the second contains just $u$ summands. Hence
\begin{align*}
  \abs{\{\abs{\sha b}>\tfrac12\lambda\}}
  \leq\Babs{\bigcup_{L\in\mathscr{B}}L}
  +\Babs{\Big\{\Babs{\sum_{L\in\mathscr{B}}\sum_{i=1}^uA_{L^{(i)}} b_L}>\tfrac12\lambda\Big\}},
\end{align*}
where the first term in bounded in the standard way by $\displaystyle\sum_{L\in\mathscr{B}}\abs{L}\lesssim\lambda^{-1}\Norm{f}{1}$.

The second term is estimated as follows:
\begin{align*}
  &\Babs{\Big\{\Babs{\sum_{L\in\mathscr{B}}\sum_{i=1}^u A_{L^{(i)}} b_L}>\tfrac12\lambda\Big\}} \\
  &\leq\frac{2}{\lambda}\BNorm{\sum_{L\in\mathscr{B}}\sum_{i=1}^u A_{L^{(i)}} b_L}{1} 
     \leq\frac{2}{\lambda}\sum_{L\in\mathscr{B}}\sum_{i=1}^u \Norm{A_{L^{(i)}} b_L}{1} \\
  &\lesssim\frac{1}{\lambda}\sum_{L\in\mathscr{B}} \sum_{i=1}^u \Norm{b_L}{1}
     \leq\frac{u}{\lambda}\sum_{L\in\mathscr{B}}\Norm{b_L}{1}
     \lesssim\frac{u}{\lambda}\Norm{f}{1},
\end{align*}
where the uniform $L^1$-boundedness of the averaging operators $A_K$ was used in the third-to-last step.
\end{proof}

\section{The weighted testing conditions in terms of shifts}

It was explained in the Introduction that the P\'erez--Treil--Volberg result \cite{PTV} reduced the proof of the $A_2$ conjecture to the verification of the testing condition
\begin{equation*}
  \Norm{T(w1_Q)}{L^2(\sigma)}\lesssim\Norm{w}{A_2}w(Q)^{1/2}
\end{equation*}
for all cubes $Q\subset\R^N$. The left side is the supremum over all normalised, compactly supported (thanks to density) $f\in L^2(w)$ of
\begin{align*}
  \pair{f}{T(w1_Q)}
  =\lim_{n\to\infty}\Exp_{\beta}\sum_{u,v=r}^{\infty} 2^{-\max(u,v)\alpha/2}\pair{f}{\sha^{uv}_{n\beta}(w1_Q)}.
\end{align*}
Therefore, it suffices to prove the corresponding testing estimate
\begin{equation*}
  \Norm{\sha^{uv}_{n\beta}(w1_Q)}{L^2(\sigma)}\lesssim \Phi(u,v)\Norm{w}{A_2}w(Q)^{1/2},
\end{equation*}
with some $\Phi(u,v)$ such that the series $\sum_{u,v=r}^{\infty}2^{-\max(u,v)\alpha/2}\Phi(u,v)$ is summable. Note that the cube $Q$ in this testing condition is completely arbitrary; it does not in general belong to the (also arbitrary) dyadic systems appearing in the definition of the dyadic shift.

The rest of the paper is dedicated to proving the following estimate, from which the required summability follows (thanks to $\alpha/2-\gamma N/2>\alpha/4>0$), thereby verifying the $A_2$ conjecture.

\begin{theorem}\label{thm:mainEst}
Let $\displaystyle\sha=\sum_{K\in\mathscr{D}}A_K$ be a good, finite, bounded dyadic shift with parameters $(u,v)$. Then
\begin{equation*}
   \Norm{\sha(w1_Q)}{L^2(\sigma)}\lesssim 2^{\max(u,v)\gamma N/2}uv\Norm{w}{A_2}w(Q)^{1/2}
\end{equation*}
for all cubes $Q\subset\R^N$. (The exponential factor is unnecessary if $Q\in\mathscr{D}$.)
\end{theorem}

As before, $A_K(w1_Q)$ can only be nonzero if $K\cap Q\neq\varnothing$, and therefore
\begin{equation*}
  \sha(w1_Q)=\sum_{K:K\cap Q\neq\varnothing}A_K(w1_Q)
  =\sum_{\substack{K:K\cap Q\neq\varnothing\\ \ell(K)\geq\ell(Q)}}A_K(w1_Q)+\sum_{\substack{K:K\cap Q\neq\varnothing\\ \ell(K)<\ell(Q)}}A_K(w1_Q).
\end{equation*}
The large scales is by far the easier part of the estimate, and in fact uniform with respect to the shift parameters:
\begin{align*}
  \Babs{\sum_{k=0}^{\infty}\sum_{\substack{K:K\cap Q\neq\varnothing\\ \ell(K)=2^k\ell(Q)}}A_K(w1_Q)}
  &\lesssim\sum_{k=0}^{\infty}\sum_{\substack{K:K\cap Q\neq\varnothing\\ \ell(K)=2^k\ell(Q)}}\frac{w(Q)}{\abs{K}}1_K \\
  &\lesssim\frac{w(Q)}{\abs{Q}}1_{3Q}+M(w1_Q) 1_{(3Q)^c}.
\end{align*}
For the first term on the right,
\begin{align*}
  \BNorm{\frac{w(Q)}{\abs{Q}}1_{3Q}}{L^2(\sigma)}
  =\frac{w(Q)}{\abs{Q}}\sigma(3Q)^{1/2}
  &\lesssim w(Q)^{1/2}\Big(\frac{w(3Q)}{\abs{3Q}} \frac{\sigma(3Q)}{\abs{3Q}}\Big)^{1/2} \\
  &\leq\Norm{w}{A_2}^{1/2}w(Q)^{1/2}.
\end{align*}
And for the second, as a direct application of Buckley's estimate \cite[Theorem~2.5]{Buckley}
\begin{equation}\label{eq:Buckley}
  \Norm{Mf}{L^2(w)}\lesssim\Norm{w}{A_2}\Norm{f}{L^2(w)},
\end{equation}
it follows that
\begin{align*}
  \Norm{M(w1_Q)}{L^2(\sigma)}
  \lesssim\Norm{\sigma}{A_2}\Norm{w1_Q}{L^2(\sigma)}=\Norm{w}{A_2} w(Q)^{1/2}.
\end{align*}

The main part of the argument consists of handling the small scales.

\section{The main estimates}

This section contains the core inequalities behind the $A_2$ conjecture. They follow quite closely the innovative estimates originally due to Lacey, Petermichl and Reguera \cite{LPR}, which gave the analogue of the $A_2$ conjecture for individual dyadic shifts. However, in order to obtain bounds with admissible dependence on the shift parameters, a number of modifications are needed here and there, so it seems appropriate to present the argument in full detail. It is also worth recalling the additional difficulty here that the cube $Q$ need not be dyadic; this is to some extent compensated by goodness of the shift under consideration, as will be apparent in the very last Lemma~\ref{lem:sumStopping} below.

With the dyadic shift of interest, $\displaystyle\sha=\sum_{K\in\mathscr{D}}A_K$, fixed for the moment, let
\begin{equation*}
  \sha_{\mathscr{C}}:=\sum_{K\in\mathscr{C}}A_K,
\end{equation*}
whenever $\mathscr{C}\subset\mathscr{D}$ is a subset. With this notation, the goal is to estimate
\begin{equation*}
  \sha_{\{K\in\mathscr{D}:K\cap Q\neq\varnothing,\ell(K)<\ell(Q)\}}(w1_Q).
\end{equation*}
In fact, since $\sha$ is good, which means that the kernel of each $A_K$ is supported only on the subset
\begin{equation*}
  \hat{K}:=\{x\in K:\dist(x,\partial K)\geq 2^{-\max(u,v)\gamma}\ell(K)\},
\end{equation*}
the condition that $A_K(w1_Q)\neq 0$ implies that even $\hat{K}\cap Q\neq 0$. Letting
\begin{equation*}
  \mathscr{K}:=\{K\in\mathscr{D}:\hat{K}\cap Q\neq\varnothing, \ell(K)<\ell(Q)\},
\end{equation*}
the task is reduced to proving that
\begin{equation}\label{eq:mainToProve}
  \Norm{\sha_{\mathscr{K}}(w1_Q)}{L^2(\sigma)}\lesssim 2^{\max(u,v)\gamma N}uv\Norm{w}{A_2}w(Q)^{1/2}.
\end{equation}

\subsection{Pigeonholing \`a la Lacey et al.}
The bound~\eqref{eq:mainToProve} will be accomplished by carefully partitioning the collection $\mathscr{K}$ into appropriate subsets, where the weights $w$ and $\sigma$ are well under control --- a procedure introduced by Lacey, Petermichl and Reguera \cite{LPR}. This consists of several steps:
\begin{enumerate}
  \item The collection $\mathscr{K}$ is partitioned into $v+1$ subcollections simply according to the value of $\log_2\ell(K)\mod v+1$. This is the step which introduces the factor $v$ into the estimate. Henceforth, an arbitrary but fixed subcollection like this will be considered, and with slight abuse still denoted by $\mathscr{K}$. Note that $A_K(w1_Q)$, which is a linear combination of Haar functions on cubes $J\in\mathscr{D}$ with $\ell(J)=2^{-v}\ell(K)$, is constant on dyadic cubes of length $2^{-v-1}\ell(K)$, and hence on all cubes $K'\in\mathscr{K}$ with $\ell(K')<\ell(K)$.

  \item The local $A_2$ characteristic is essentially fixed by considering the subsets $\mathscr{K}^a$ of those $K\in\mathscr{K}$ with
  \begin{equation*}
  2^{a}<\frac{w(K\cap Q)}{\abs{K}}\cdot\frac{\sigma(K)}{\abs{K}}\leq 2^{a+1},
\end{equation*}
where $a\in\Z$ with $a\leq \log_2\Norm{w}{A_2}$.

  \item Among each $\mathscr{K}^a$, a subset of stopping cubes $\mathscr{S}^a=\bigcup_{k=0}^{\infty}\mathscr{S}^a_k$ is constructed as follows: $\mathscr{S}^a_0$ consists of all maximal (with respect to set inclusion) $K\in\mathscr{K}^a$, and then inductively $\mathscr{S}^a_{k+1}$ consists of all maximal $K\in\mathscr{K}^a$ such that
  \begin{equation*}
  \frac{w(K\cap Q)}{\abs{K}}>4\frac{w(S\cap Q)}{\abs{S}}
\end{equation*}
for some $S\in\mathscr{S}^a_k$ with $S\supset K$. For $K\in\mathscr{K}^a$, let $K^s$ stand for the minimal stopping cube $S\in\mathscr{S}^a$ with $S\supseteq K$. Then the collections
\begin{equation*}
  \mathscr{K}^a(S):=\{K\in\mathscr{K}^a: K^s=S\},\qquad S\in\mathscr{S}^a,
\end{equation*}
form a partition of $\mathscr{K}^a$. (Constructions of this type are known in the literature under different names, including ``principal cubes'' and ``corona decompositions.'')

  \item Finally, yet another measure ratio is essentially fixed by considering the subcollections $\mathscr{K}^a_b(S)$ of those $K\in\mathscr{K}^a(S)$ with
  \begin{equation*}
  2^{1-b}\frac{w(S\cap Q)}{\abs{S}}<\frac{w(K\cap Q)}{\abs{K}}\leq 2^{2-b}\frac{w(S\cap Q)}{\abs{S}},\qquad b\in\N.
\end{equation*}
Note that for $K\in\mathscr{K}^a_b(S)$, there holds
\begin{equation*}
   \frac{\sigma(K)}{\abs{K}}\eqsim 2^a\frac{\abs{K}}{w(K\cap Q)}\eqsim 2^{a+b}\frac{\abs{S}}{w(S\cap Q)}=:\tau^a_b(S),
\end{equation*}
so the $\sigma$ and Lebesgue measures are essentially comparable, with their ratio depending only on $a,b$ and $S$.
\end{enumerate}

The proof of \eqref{eq:mainToProve} then starts by writing
\begin{align*}
  &\Norm{\sha_{\mathscr{K}}(w1_Q)}{L^2(\sigma)} \\
  &\leq\sum_{a:2^a\leq\Norm{w}{A_2}}\Big(\int\Babs{\sum_{S\in\mathscr{S}^a}\sha_{\mathscr{K}^a(S)}(w1_Q)}^2\sigma \Big)^{1/2} \\
  &\leq \sum_{a:2^a\leq\Norm{w}{A_2}}\Big(\sum_{S\in\mathscr{S}^a}\int\abs{\sha_{\mathscr{K}^a(S)}(w1_Q)}^2\sigma \\
  &\phantom{\leq \sum_{a:2^a\leq\Norm{w}{A_2}}\Big(}+2\sum_{S\in\mathscr{S}^a}\sum_{\substack{S'\in\mathscr{S}^a\\ S'\subset S}}
        \int\abs{\sha_{\mathscr{K}^a(S)}(w1_Q)}\cdot\abs{\sha_{\mathscr{K}^a(S')}(w1_Q)}\sigma\Big)^{1/2}.
\end{align*}
It is further observed that all $K\in\mathscr{K}^a(S)$ are either disjoint from or strictly containing any $S'\in\mathscr{S}^a$ with $S'\subset S$; hence all these $A_K(w1_Q)$, and thus $\sha_{\mathscr{K}^a(S)}(w1_Q)$ itself, are constant on $S'$. Thus
\begin{align*}
   \int &\abs{\sha_{\mathscr{K}^a(S)}(w1_Q)}\cdot\abs{\sha_{\mathscr{K}^a(S')}(w1_Q)}\sigma \\
    &=\abs{\ave{\sha_{\mathscr{K}^a(S)}(w1_Q)}_{S'}}\int\abs{\sha_{\mathscr{K}^a(S')}(w1_Q)}\sigma
\end{align*}
The next task is obtaining useful bounds for the integral on the right.

\subsection{John--Nirenberg-type estimates}
The goal is to estimate the size of the set where
\begin{equation*}
   \abs{\sha_{\mathscr{K}^a(S)}(w1_Q)}>t, 
\end{equation*}
both with respect to the Lebesgue and $\sigma$ measures. The available information is the weak-type $L^1$ bound for the dyadic shifts, and the Lebesgue measure estimate could be deduced directly from this by a usual John--Nirenberg-type argument. However, in order to smoothen the passage to the $\sigma$ measure estimate, it is useful to first consider the shifts restricted to the collections $\mathscr{K}^a_b(S)$, where the two measures are comparable.

\begin{lemma}
For a good, finite, bounded dyadic shift $\sha$ with parameters $(u,v)$, the following estimates hold when $\nu$ is either the Lebesgue or the $\sigma$ measure:
\begin{equation*}
  \nu\Big(\Big\{\abs{\sha_{\mathscr{K}^a_b(S)}(w1_Q)}>u2^{-b}\frac{w(S\cap Q)}{\abs{S}}\cdot t\Big\}\Big)
  \lesssim e^{-ct}\nu(S),\qquad t\geq 0,
\end{equation*}
where $c>0$ is a constant.
\end{lemma}

\begin{proof}
Let $\lambda:=Cu2^{-b} w(S\cap Q)/\abs{S}$, where $C$ is a large constant, and $n\in\Z_+$. Let $x\in\R^N$ be a point where
\begin{equation}\label{eq:>nLambda}
  \abs{\sha_{\mathscr{K}^a_b(S)}(w1_Q)(x)}>n\lambda.
\end{equation}
Then for all small enough $L\in\mathscr{K}^a_b(S)$ with $L\owns x$, there holds
\begin{equation*}
  \Babs{\sum_{\substack{K\in\mathscr{K}^a_b(S) \\ K\supseteq L}}A_K(w1_Q)(x)}>n\lambda.
\end{equation*}
Since $\displaystyle\sum_{\substack{K\in\mathscr{K}^a_b(S) \\ K\supset L}}A_K(w1_Q)(x)$ is constant on $L$, and
\begin{equation}\label{eq:ALwQ}
  \Norm{A_L(w1_Q)}{\infty}\lesssim\frac{w(L\cap Q)}{\abs{L}}\lesssim 2^{-b}\frac{w(S\cap Q)}{\abs{S}},
\end{equation}
it follows that
\begin{equation}\label{eq:>n-1/3}
   \Babs{\sum_{\substack{K\in\mathscr{K}^a_b(S) \\ K\supset L}}A_K(w1_Q)}>(n-\tfrac{1}{3})\lambda\qquad\text{on }L.
\end{equation}
Let $\mathscr{L}\subseteq\mathscr{K}^a_b(S)$ be the collection of maximal cubes with the above property. Thus all $L\in\mathscr{L}$ are disjoint, and all $x$ with \eqref{eq:>nLambda} belong to some $L$. By maximality of $L$, the minimal $L^*\in\mathscr{K}^a_b(S)$ with $L^*\supset L$ satisfies
\begin{equation*}
   \Babs{\sum_{\substack{K\in\mathscr{K}^a_b(S) \\ K\supset L^*}}A_K(w1_Q)}\leq(n-\tfrac{1}{3})\lambda\qquad\text{on }L^*.
\end{equation*}
By an estimate similar to \eqref{eq:ALwQ}, with $L^*$ in place of $L$, it follows that
\begin{equation*}
  \Babs{\sum_{\substack{K\in\mathscr{K}^a_b(S) \\ K\supset L}}A_K(w1_Q)}\leq (n-\tfrac{2}{3})\lambda\qquad\text{on }L.
\end{equation*}
Thus, if $x$ satisfies \eqref{eq:>nLambda} and $x\in L\in\mathscr{L}$, then necessarily
\begin{equation*}
  \abs{\sha_{\{K\in\mathscr{K}^a_b(S); K\subseteq L\}}(w1_{Q\cap L})(x)}=
  \Babs{\sum_{\substack{K\in\mathscr{K}^a_b(S) \\ K\subseteq L}}A_K(w1_Q)(x)}>\tfrac13\lambda.
\end{equation*}
Using the weak-type $L^1$ estimate, which is uniform over all bounded dyadic shifts with parameters $(u,v)$, it follows that
\begin{align*}
  \Babs{\Big\{\Babs{\sum_{\substack{K\in\mathscr{K}^a_b(S) \\ K\subseteq L}}A_K(w1_Q)(x)}>\tfrac13\lambda\Big\}}
  &\leq \frac{Cu}{\lambda}w(L\cap Q) \\
  &\leq\frac{Cu}{\lambda}2^{-b}\frac{w(S\cap Q)}{\abs{S}}\abs{L} \leq \tfrac13\abs{L},
\end{align*}
provided that the constant in the definition of $L$ was chosen large enough. Recalling \eqref{eq:>n-1/3}, there holds
\begin{align*}
  \Babs{\sum_{K\in\mathscr{K}^a_b(S)}A_K(w1_Q)}
  &\geq\Babs{\sum_{\substack{K\in\mathscr{K}^a_b(S) \\ K\supset L}}A_K(w1_Q)}-\Babs{\sum_{\substack{K\in\mathscr{K}^a_b(S) \\ K\subseteq L}}A_K(w1_Q)} \\
  &>(n-\tfrac13)\lambda-\tfrac13\lambda=(n-\tfrac23)\lambda\quad\text{on }\tilde{L}\subset L\text{ with }\abs{\tilde{L}}\geq\tfrac23\abs{L}.
\end{align*}
Thus
\begin{align*}
  \abs{\{\abs{\sha_{\mathscr{K}^a_b(S)}(w1_Q)}>n\lambda\}}
  &\leq\sum_{L\in\mathscr{L}}\abs{L\cap \{\abs{\sha_{\mathscr{K}^a_b(S)}(w1_Q)}>n\lambda\}} \\
  &\leq\sum_{L\in\mathscr{L}}\abs{\{\abs{\sha_{\{K\in\mathscr{K}^a_b(S):K\subseteq L\}}(w1_Q)}>\tfrac13\lambda\}} \\
  &\leq\sum_{L\in\mathscr{L}}\tfrac13\abs{L}\leq\sum_{L\in\mathscr{L}}\tfrac13\cdot\tfrac 32\abs{\tilde{L}} \\
  &\leq\tfrac12\sum_{L\in\mathscr{L}} \abs{L\cap\{ \abs{\sha_{\mathscr{K}^a_b(S)}(w1_Q)}>(n-1)\lambda\}} \\
  &\leq\tfrac12\abs{\{ \abs{\sha_{\mathscr{K}^a_b(S)}(w1_Q)}>(n-1)\lambda\}}.
\end{align*}
By induction it follows that
\begin{align*}
  \abs{\{\abs{\sha_{\mathscr{K}^a_b(S)}(w1_Q)}>n\lambda\}}
  &\leq 2^{-n}\abs{\{ \abs{\sha_{\mathscr{K}^a_b(S)}(w1_Q)}>0\}} \\
  &\leq 2^{-n}\sum_{M\in\mathscr{M}}\abs{M}\leq 2^{-n}\abs{S},
\end{align*}
where $\mathscr{M}$ is the collection of maximal cubes in $\mathscr{K}^a_b(S)$.

To deduce the corresponding estimate for the $\sigma$ measure, selected intermediate steps of the above computation, as well as the definition of $\mathscr{K}^a_b(S)$, will be exploited:
\begin{align*}
  \sigma(\{\abs{\sha_{\mathscr{K}^a_b(S)}(w1_Q)}>n\lambda\})
  &\leq\sum_{L\in\mathscr{L}}\sigma(L) 
    \lesssim\sum_{L\in\mathscr{L}}\tau^a_b(S)\abs{L} \\
  &\lesssim\tau^a_b(S)\abs{\{ \abs{\sha_{\mathscr{K}^a_b(S)}(w1_Q)}>(n-1)\lambda\}} \\
  &\lesssim \tau^a_b(S)2^{-n}\sum_{M\in\mathscr{M}}\abs{M} \\
  &\lesssim 2^{-n}\sum_{M\in\mathscr{M}}\sigma(M)\leq 2^{-n}\sigma(S).\qedhere
\end{align*}
\end{proof}

It is an immediate consequence that a similar estimate holds for the bigger collections $\displaystyle\mathscr{K}^a(S)=\bigcup_{b=0}^{\infty}\mathscr{K}^a_b(S)$; indeed
\begin{align*}
    \nu\Big(\Big\{ &\abs{\sha_{\mathscr{K}^a(S)}(w1_Q)}>u\frac{w(S\cap Q)}{\abs{S}}\cdot t\Big\}\Big) \\
    &\leq\sum_{b=0}^{\infty}\nu\Big(\Big\{\abs{\sha_{\mathscr{K}^a_b(S)}(w1_Q)}>u2^{-b}\frac{w(S\cap Q)}{\abs{S}}\cdot c2^{b/2} t\Big\}\Big) \\
    &\lesssim\sum_{b=0}^{\infty} e^{-c2^{b/2}t}\nu(S)\lesssim \sum_{b=0}^{\infty} e^{-c2^{b/2}}e^{-ct}\nu(S)\lesssim e^{-ct}\nu(S),
\end{align*}
where the computation is valid at least for $t\geq 2$, and the conclusion is trivial otherwise. The final conclusion, for both measures, is that
\begin{equation}\label{eq:LpBound}
  \int\abs{\sha_{\mathscr{K}^a(S)}(w1_Q)}^p\ud\nu
  \lesssim\Big(u\frac{w(S\cap Q)}{\abs{S}}\Big)^p\nu(S),\qquad p\in[1,\infty).
\end{equation}

\subsection{Conclusion of the proof}
Returning to the estimation of $\Norm{\sha_{\mathscr{K}}(w1_Q)}{L^2(\sigma)}$, it has so far been shown that
\begin{align*}
  &\Norm{\sha_{\mathscr{K}}(w1_Q)}{L^2(\sigma)} \\
  &\leq \sum_{2^a\leq\Norm{w}{A_2}}\Big(\sum_{S\in\mathscr{S}^a}\int\abs{\sha_{\mathscr{K}^a(S)}(w1_Q)}^2\sigma \\
  &\phantom{\leq \sum_{a:2^a\leq\Norm{w}{A_2}}\Big(}+2\sum_{S\in\mathscr{S}^a}\sum_{\substack{S'\in\mathscr{S}^a\\ S'\subset S}}
         \abs{\ave{\sha_{\mathscr{K}^a(S)}(w1_Q)}_{S'}}\int\abs{\sha_{\mathscr{K}^a(S')}(w1_Q)}\sigma\Big)^{1/2}.
\end{align*}
Substituting the estimate \eqref{eq:LpBound} with $\nu=\sigma$ and $p=1,2$, this continues with
\begin{align*}
  &\lesssim \sum_{2^a\leq\Norm{w}{A_2}}\Big(\sum_{S\in\mathscr{S}^a}\Big(u\frac{w(S\cap Q)}{\abs{S}}\Big)^2\sigma(S) \\
  &\phantom{\leq \sum_{a:2^a\leq\Norm{w}{A_2}}\Big(}+\sum_{S\in\mathscr{S}^a}\sum_{\substack{S'\in\mathscr{S}^a\\ S'\subset S}}
         \abs{\ave{\sha_{\mathscr{K}^a(S)}(w1_Q)}_{S'}}\Big(u\frac{w(S'\cap Q)}{\abs{S'}}\Big)\sigma(S')\Big)^{1/2},
\end{align*}
and recalling the freezing of the local $A_2$ characteric in the definition of $\mathscr{K}^a$,
\begin{align*}         
  &\lesssim \sum_{2^a\leq\Norm{w}{A_2}} 2^{a/2}\Big(u^2\sum_{S\in\mathscr{S}^a} w(S\cap Q)
    +u\sum_{S\in\mathscr{S}^a}\sum_{\substack{S'\in\mathscr{S}^a\\ S'\subset S}}
         \abs{\ave{\sha_{\mathscr{K}^a(S)}(w1_Q)}_{S'}}\,\abs{S'}\Big)^{1/2}.  
\end{align*}

Concentrating for the moment on the last term,
\begin{align*}
   \sum_{\substack{S'\in\mathscr{S}^a\\ S'\subset S}} &\abs{\ave{\sha_{\mathscr{K}^a(S)}(w1_Q)}_{S'}}\,\abs{S'}
   \leq \sum_{\substack{S'\in\mathscr{S}^a\\ S'\subset S}}\int_{S'}\abs{\sha_{\mathscr{K}^a(S)}(w1_Q)}\ud x \\
   &=\int\Big(\sum_{\substack{S'\in\mathscr{S}^a\\ S'\subset S}}1_{S'}\Big)\abs{\sha_{\mathscr{K}^a(S)}(w1_Q)}\ud x \\
   &\leq\BNorm{\sum_{\substack{S'\in\mathscr{S}^a\\ S'\subset S}}1_{S'}}{L^2}\Norm{\sha_{\mathscr{K}^a(S)}(w1_Q)}{L^2}.
\end{align*}
The first factor is bounded by $\abs{S}^{1/2}$, as one easily checks from the construction of the stopping cubes: those $S'\subset S$ of the first generation are disjoint, and
\begin{equation*}
  \sum_{S'}\abs{S'}\leq\sum_{S'}\frac{1}{4}w(S'\cap Q)\frac{\abs{S}}{w(S\cap Q)}
  \leq\frac14 w(S\cap Q)\frac{\abs{S}}{w(S\cap Q)}=\frac14\abs{S};
\end{equation*}
one simply repeats this for the consecutive generations and sums up a geometric series. The second factor may be estimated by \eqref{eq:LpBound} with the Lebesgue measure and $p=2$, to the result that
\begin{equation*}
   \Norm{\sha_{\mathscr{K}^a(S)}(w1_Q)}{L^2}\lesssim\Big(u\frac{w(S\cap Q)}{\abs{S}}\Big)\abs{S}^{1/2}.
\end{equation*}
Thus, altogether
\begin{equation*}
  \sum_{\substack{S'\in\mathscr{S}^a\\ S'\subset S}} \abs{\ave{\sha_{\mathscr{K}^a(S)}(w1_Q)}_{S'}}\,\abs{S'}
  \lesssim u\cdot w(S\cap Q),
\end{equation*}
and then
\begin{equation*}
  \Norm{\sha_{\mathscr{K}}(w1_Q)}{L^2(\sigma)}
  \lesssim u\sum_{2^a\leq\Norm{w}{A_2}}2^{a/2} \Big(\sum_{S\in\mathscr{S}^a}w(S\cap Q)\Big)^{1/2}.
\end{equation*}
The proof is completed by the following lemma, for then
\begin{align*}
   \Norm{\sha_{\mathscr{K}}(w1_Q)}{L^2(\sigma)}
  &\lesssim u\sum_{2^a\leq\Norm{w}{A_2}}2^{a/2}\Big(2^{\max(u,v)\gamma N}\Norm{w}{A_2}w(Q)\Big)^{1/2} \\
  &\lesssim u2^{\max(u,v)\gamma N/2}\Norm{w}{A_2}w(Q)^{1/2};
\end{align*}
recall that the final estimate will also involve the factor $v$ resulting from summing up the $v+1$ subcollections in the first step of the pigeonholing.

\begin{lemma}\label{lem:sumStopping}
\begin{equation*}
  \sum_{S\in\mathscr{S}^a}w(S\cap Q)\lesssim 2^{\max(u,v)\gamma N}\Norm{w}{A_2}w(Q).
\end{equation*}
\end{lemma}

\begin{proof}
Recall the notation $\hat{K}$ from the beginning of this section, right before \eqref{eq:mainToProve}.
Every $K\in\mathscr{K}$ satisfies $\hat{K}\cap Q\neq\varnothing$ and $\ell(K)<\ell(Q)$, which imply that $K\cap Q$ must contain a cube of sidelength $2^{-\max(u,v)\gamma}\ell(K)$, thus of volume $2^{-\max(u,v)\gamma N}\abs{K}$. This holds in particular for every $S\in\mathscr{S}^a\subseteq\mathscr{K}$. Hence
\begin{align*}
  \sum_{S\in\mathscr{S}^a}w(S\cap Q) &\leq 2^{\max(u,v)\gamma N}\sum_{S\in\mathscr{S}^a}\frac{w(S\cap Q)}{\abs{S}}\abs{S\cap Q} \\
  &=2^{\max(u,v)\gamma N}\int_Q \sum_{S\in\mathscr{S}^a}\frac{w(S\cap Q)}{\abs{S}}1_S(x)\ud x.
\end{align*}
For a fixed point $x$, the construction of the stopping cubes ensures that the ratio $w(S\cap Q)/\abs{S}$ along $S\owns x$ increases at least geometrically, and hence their sum is dominated by the maximal value, which in turn is dominated by $M(w1_Q)(x)$. Thus
\begin{align*}
   \int_Q \sum_{S\in\mathscr{S}^a}\frac{w(S\cap Q)}{\abs{S}}1_S(x)\ud x
   &\lesssim\int_Q M(w1_Q)\ud x
   \leq\Norm{M(w1_Q)}{L^2(\sigma)}\Norm{1_Q}{L^2(w)} \\
   &\lesssim\Norm{\sigma}{A_2}\Norm{w1_Q}{L^2(\sigma)}w(Q)^{1/2}=\Norm{w}{A_2}w(Q)
\end{align*}
by an application of Buckley's estimate \eqref{eq:Buckley}.
\end{proof}

Note that if $Q\in\mathscr{D}$, then all $K\in\mathscr{K}$ satisfy $K\subseteq Q$, hence $K\cap Q=K$, and the introduction of the exponential factor, as well as the use of the goodness of the shift at this point, is unnecessary.

\section{Discussion}

\subsection{A shorter proof of the $A_2$ conjecture?}
At the present, a self-contained proof of the $A_2$ conjecture would consist of the almost 40 pages of P\'erez, Treil and Volberg's reduction to the weak-type estimate \cite{PTV}, combined with the present argument to provide this last missing information. It is perhaps interesting that both these steps go through a $T(1)$ theorem for a Calder\'on--Zygmund operator, and using a Haar wavelet basis; however, one adapted to the measures $w$ and $\sigma$ in P\'erez--Treil--Volberg's part \cite{PTV}, and the standard one in the present contribution.

While it gives the desired result, this combination might be a bit of overshooting: since the present argument already reduces things to the dyadic shift operators, it should philosophically be enough to use a weight-adapted $T(1)$-theorem for these shifts, rather than for general Calder\'on--Zygmund operators. And for dyadic operators, it should ideally be enough to verify the weighted testing condition for dyadic cubes only, which would somewhat simplify the preceding analysis. Indeed, a result of this flavour is provided by Nazarov--Treil--Volberg's two-weight inequality for dyadic shift operators \cite{NTV:2weightHaar} (which lies behind Lacey--Petermichl--Reguera's result \cite{LPR}). But in order to apply it to the desired conclusion, one would need to keep track of the dependence of their estimate on the shift parameters, to ensure the required summability in the end, whereas the P\'erez--Treil--Volberg result \cite{PTV} may be directly applied as a black box.

It would also be interesting if the Lerner's formula -based Cruz-Uribe--Martell--P\'erez approach \cite{CMP} to the Lacey--Petermichl--Reguera estimate \cite{LPR} could be improved so as to have summable dependence on the shift parameters. 

\subsection{Possible extensions}
The representation of a Calder\'on--Zygmung operator as an average of good dyadic shifts is an identity, which has no specific connection to $A_2$ weights, and may be useful for proving other bounds as well. In particular, it is likely that the same proof strategy is also applicable to providing sharp weighted weak-type $L^p$ bounds for general Calder\'on--Zygmund operators, in a similar way as the the Lacey--Petermichl--Reguera argument was extended to weak-type $L^p$ bounds for dyadic shifts \cite{HLRV} and smooth Calder\'on--Zygmund operators \cite{HLRSUV} by Lacey et al. This would involve verifying the weak-type testing condition of Lacey--Sawyer--Uriarte-Tuero \cite{LSU}, which is very similar to the P\'erez--Treil--Volberg testing condition \cite{PTV} checked in this paper. The main difference is that the Lacey--Sawyer--Uriarte-Tuero condition requires the estimation of the maximal truncations of $T$, rather than just the operator itself; on the other hand, the conclusions of their theorem are then valid for the maximal truncations as well.

P\'erez, Treil and Volberg assert that their result extends to Calder\'on--Zygmund operators on spaces of homogeneous type \cite[Section~12]{PTV}. It is likely that the present argument will do so as well. In particular, the dyadic cubes in this generality have already been constructed by Christ \cite{Christ}, and the required randomisation of this construction was recently carried out by Martikainen and the author \cite{HM}. The present arguments also made use of some specific symmetries of the Euclidean space, especially the fact that the probability of a cube being good is constant. A trick to ensure this even in a metric space has been presented by Martikainen \cite{Martikainen}. One would still need to check whether the computation of the conditional probabilities, which here employed the explicit form of the randomisation in terms of the binary variables $\beta_j$, is compatible with the abstract randomisation procedure in a metric space. The actual estimates for the shifts above mainly relied on the abstract dyadic structure, and would probably extend reasonably straightforwardly.

\bibliography{weighted}
\bibliographystyle{plain}

\end{document}